\documentclass[reqno,11pt,centertags]{amsart}
\usepackage{amsmath,amsthm,amscd,amssymb,latexsym,esint,upref,stmaryrd,enumerate,verbatim,yfonts,bm,mathtools,comment}
\usepackage{color}
\usepackage{mathrsfs}
\usepackage[margin=1.5in]{geometry}

\usepackage{hyperref}
 
\newcommand*{\mailto}[1]{\href{mailto:#1}{\nolinkurl{#1}}}

\newcommand{\bbN}{{\mathbb{N}}}

\newcommand{\bbR}{{\mathbb{R}}}
\newcommand{\bbC}{{\mathbb{C}}}
\newcommand{\bbZ}{{\mathbb{Z}}}

\newcommand{\sH}{\mathscr{H}}

\DeclareMathOperator{\rank}{rank}

\DeclareMathOperator{\dom}{dom}
\newcommand{\dott}{\,\cdot\,}

\newcommand{\Lp}{\text{\textnormal{L}}}
\newcommand{\AC}{\text{\textnormal{AC}}}

\newcommand{\no}{\notag}
\newcommand{\lb}{\label}

\newcommand{\bi}{\bibitem}
 
\newcommand{\Hz}[1]{H_{Z,{#1}}}

\newcommand{\vhi}{\varphi}

\makeatletter
\def\theequation{\@arabic\c@equation}
 
\allowdisplaybreaks
\numberwithin{equation}{section}
 
\newtheorem{theorem}{Theorem}[section]
\newtheorem{proposition}[theorem]{Proposition}
\newtheorem{lemma}[theorem]{Lemma}

\newtheorem{hypothesis}[theorem]{Hypothesis}

\theoremstyle{definition}
\newtheorem{definition}[theorem]{Definition}
\theoremstyle{remark}
\newtheorem{remark}[theorem]{Remark}

\begin{document}
 
\title[The Krein--von Neumann Extension of a Quasi-Differential Operator]{The Krein--von Neumann Extension of a Regular Even Order Quasi-Differential Operator}
 
 
\author[M.\ Cho]{Minsung Cho}
\address{Department of Mathematical Sciences,
Carnegie Mellon University, 5000 Forbes Avenue,
Pittsburgh, PA 15289, USA}
\email{\mailto{minsungc@andrew.cmu.edu}}
\urladdr{\url{http://www.andrew.cmu.edu/user/minsungc}}
\author[S.\ Hoisington]{Seth Hoisington}
\address{Department of Mathematics,
University of Virginia, Charlottesville, VA 22903, USA}
\email{\mailto{seh9gb@virginia.edu}}
\author[R.\ Nichols]{Roger Nichols}
\address{Department of Mathematics (Dept.~6956), The University of Tennessee at Chattanooga, 615 McCallie Ave., Chattanooga, TN 37403, USA}
\email{\mailto{Roger-Nichols@utc.edu}}
\urladdr{\url{https://sites.google.com/mocs.utc.edu/rogernicholshomepage/home}}
\author[B.\ Udall]{Brian Udall}
\address{Department of Mathematics,
Rice University,
6100 Main Street, Houston, TX 77005, USA}
\email{\mailto{bu3@rice.edu}}
 
 
\date{\today}
 
\@namedef{subjclassname@2020}{\textup{2020} Mathematics Subject Classification}

\begin{abstract}
We characterize by boundary conditions the Krein--von Neumann extension of a strictly positive minimal operator corresponding to a regular even order quasi-differential expression of Shin--Zettl type.  The characterization is stated in terms of a specially chosen basis for the kernel of the maximal operator and employs a description of the Friedrichs extension due to M\"oller and Zettl.
\end{abstract}
 
\maketitle
 
 

\section{Introduction} \lb{s1}
 
A linear operator $S$ acting in a separable Hilbert space $(\sH,\langle\dott,\dott\rangle_{\sH})$ with dense domain $\dom(S)$ is said to be {\it nonnegative} if
\begin{equation}\lb{1.1}
\langle u,Su\rangle_{\sH}\geq 0,\quad u\in\dom(S).
\end{equation}
If $S$ satisfies the stronger condition that for some $\varepsilon \in (0,\infty)$,
\begin{equation}\lb{1.2}
\langle u,Su\rangle_{\sH} \geq \varepsilon \langle u,u\rangle_{\sH},\quad u\in \dom(S),
\end{equation}
then $S$ is said to be {\it strictly positive} and one writes $S\geq \varepsilon I_{\sH}$, where $I_{\sH}$ denotes the identity operator in $\sH$.  The condition \eqref{1.1} implies that $S$ is {\it symmetric},
\begin{equation}\lb{1.3}
\langle u,Sv\rangle_{\sH} = \langle Su,v\rangle_{\sH},\quad u,v\in \dom(S),
\end{equation}
and that
\begin{equation}\lb{1.4}
\dim(\ker(S^*-zI_{\sH}))\in \bbN_0\cup\{\infty\}
\end{equation}
is constant with respect to $z\in \bbC\backslash [0,\infty)$.  Here $S^*$ denotes the Hilbert space adjoint of $S$.  (If $S\geq \varepsilon I_{\sH}$ for some $\varepsilon \in (0,\infty)$, then the dimension \eqref{1.4} is constant with respect to $z\in \bbC\backslash [\varepsilon,\infty)$.)  In particular, the deficiency indices of $S$ are equal, and $S$ possesses a self-adjoint extension by von Neumann's theory of self-adjoint extensions \cite{Ne30}.  We shall assume that $S$ is unbounded with nonzero deficiency indices.  Otherwise, $S$ is essentially self-adjoint, meaning the closure of $S$, which we denote by $\overline{S}$, is the only self-adjoint extension of $S$.  In addition, since the self-adjoint extensions of $S$ and those of $\overline{S}$ are the same, we shall henceforth assume that the operator $S$ is closed.
 
If $S$ is nonnegative, then its {\it Friedrichs extension} $S_{\rm F}$ is constructed in a canonical way using form methods---a classic construction that goes back to the 1934 work of Friedrichs \cite{Fr30}.  For details of the construction, we refer to \cite[Section VI.2.3]{Ka80}, \cite[Section 10.4]{Sc12}, and \cite[Theorem 2.13]{Te14}.  One important characteristic of the Friedrichs extension is that $S_{\rm F}$ has the same lower bound as the symmetric operator $S$ (see \cite[Theorem 10.17(i)]{Sc12}).  Therefore, $S_{\rm F}$ is a nonnegative self-adjoint extension of $S$.
 
In his seminal work on nonnegative self-adjoint extensions, M.~G.~Krein \cite{Kr47}, \cite{Kr47a} showed that, among all nonnegative self-adjoint extensions of $S$, there exist two which are the {\it largest} and {\it smallest}---in the sense of order between nonnegative self-adjoint operators---such extensions of $S$.  Recall that if $A$ and $B$ are nonnegative self-adjoint operators in $(\sH,\langle\dott,\dott\rangle_{\sH})$, then $A\leq B$ if and only if
\begin{equation}\lb{1.5}
\big\langle u, (B+aI_{\sH})^{-1}u\big\rangle_{\sH} \leq \big\langle u, (A+aI_{\sH})^{-1}u\big\rangle_{\sH},\quad u\in \sH,\, a\in (0,\infty).
\end{equation}
The largest nonnegative self-adjoint extension of $S$ is the Friedrichs extension $S_{\rm F}$, and the smallest nonnegative self-adjoint extension of $S$, which we shall denote by $S_{\rm K}$, is known as the {\it Krein--von Neumann extension}.  Krein's result may be summarized as follows.
 
\begin{theorem}[\cite{Kr47}]\lb{T1.1}
If $S$ is a densely defined, closed, nonnegative operator in a separable Hilbert space $(\sH,\langle\dott,\dott\rangle_{\sH})$, then there exist two nonnegative self-adjoint extensions, $S_{\rm F}$ and $S_{\rm K}$, of $S$ which are the largest and smallest, respectively, nonnegative self-adjoint extensions of $S$.  A nonnegative self-adjoint operator $S'$ is a self-adjoint extension of $S$ if and only if
\begin{equation}\lb{1.6}
S_{\rm K} \leq S' \leq S_{\rm F}.
\end{equation}
The operators $S_{\rm F}$ and $S_{\rm K}$ are uniquely determined by \eqref{1.6}.  If, in addition, $S\geq \varepsilon I_{\sH}$ for some $\varepsilon \in (0,\infty)$, then $S_{\rm F}\geq \varepsilon I_{\sH}$ and
\begin{align}
\dom(S_{\rm F})&= \dom(S) \dotplus (S_{\rm F})^{-1}\ker(S^*),\lb{1.7}\\
\dom(S_{\rm K})&= \dom(S) \dotplus \ker(S^*),\lb{1.8}
\end{align}
where $\dotplus$ denotes the direct sum of subspaces in $\sH$.
\end{theorem}
\noindent
For additional details, especially in connection with the Krein--von Neumann extension, we refer to the survey \cite{AGMST13}.

When $S$ is an ordinary differential operator, its self-adjoint extensions are usually characterized in terms of appropriate boundary conditions at the endpoints of the underlying interval.  Therefore, it is natural to try to determine the boundary conditions that characterize $S_{\rm F}$ and $S_{\rm K}$.  

Considerable attention has been given to the problem of identifying the boundary conditions that characterize $S_{\rm F}$; see, for example, \cite{MZ00}, \cite{MZ95a}, \cite{MZ95b}, \cite{NZ90}, \cite{YSZ15}, and \cite[Section 10.5]{Ze05}, to name only a few.  In particular, when $S=H_{Z,\min}$, where $H_{Z,\min}$ is the minimal operator generated by a regular even order Shin--Zettl quasi-differential expression $\tau_{{}_Z}$ (with matrix-valued coefficients) on the interval $[a,b]$, M\"oller and Zettl \cite{MZ95a} proved that $H_{Z,\min}$ is bounded from below if the leading coefficient of $\tau_{{}_Z}$ is positive definite almost everywhere.  In this case, if $\tau_{{}_Z}$ is of order $2N$ (for some $N\in \bbN$), then M\"oller and Zettl showed in \cite[Theorem 8.1]{MZ95a} that the Friedrichs extension $H_{Z,\rm F}$ of $H_{Z,\min}$ is characterized by boundary conditions on the first $N-1$ quasi-derivatives at the interval endpoints:
\begin{equation}
    y^{[j-1]}(a)=y^{[j-1]}(b)=0,\quad 1\leq j\leq N.
\end{equation}
This characterization extends an earlier result by Niessen and Zettl (see \cite[Theorem 2.1]{NZ90}).  The quasi-differential expression $\tau_{{}_Z}$ studied by M\"oller and Zettl is of a very general form and includes, as a special case, the classic quasi-differential expression studied by Naimark \cite{Na68}.

In recent years, more attention has been given to the problem of determining the boundary conditions that characterize $S_{\rm K}$.  The authors of \cite{CGNZ14} considered a regular three-coefficient Sturm--Liouville differential expression $\tau_{p,q,r}$ acting according to
\begin{equation}
    \tau_{p,q,r} f = \frac{1}{r}\left[-(f^{[1]})'+qf \right]\quad \text{on $[a,b]$},
\end{equation}
where $f^{[1]}=pf'$ denotes the quasi-derivative of $f$, $p>0$, $r>0$, $q$ is real-valued almost everywhere on the interval $[a,b]$, and $p^{-1}$, $q$, $r$ are integrable on $[a,b]$. Assuming the minimal operator $H_{p,q,r,\min}$ generated by $\tau_{p,q,r}$ is strictly positive, the Krein--von Neumann extension of $H_{p,q,r,\min}$ was characterized in terms of a specially chosen basis for the kernel of $(H_{p,q,r,\min})^*$.  Specifically, taking as a basis for the kernel of $(H_{p,q,r,\min})^*$ the set $\{u_1,u_2\}$ determined by
\begin{equation}\lb{1.11b}
    u_1(a)=u_2(b)=1,\quad u_1(b)=u_2(a)=0,
\end{equation}
it was shown in \cite[Example 3.3]{CGNZ14} that the Krein--von Neumann extension of $H_{p,q,r,\min}$ corresponds to coupled boundary conditions of the form:
\begin{equation}\lb{1.12b}
    \begin{pmatrix}
        f(b)\\
        f^{[1]}(b)
    \end{pmatrix}=
    \frac{1}{u_2^{[1]}(a)}
    \begin{pmatrix}
        -u_1^{[1]}(a) & 1\\
        u_2^{[1]}(a)u_1^{[1]}(b)-u_2^{[1]}(b)u_1^{[1]}(a) & u_2^{[1]}(b)\\
    \end{pmatrix}
    \begin{pmatrix}
        f(a)\\
        f^{[1]}(a)
    \end{pmatrix}.
\end{equation}

In \cite{EGNT13}, the characterization in \eqref{1.11b}--\eqref{1.12b} was extended to the generalized four-coefficient Sturm--Liouville differential expression $\tau_{p,q,r,s}$ acting according to
\begin{equation}
    \tau_{p,q,r,s} f = \frac{1}{r}\left[-(f^{[1]})'+sf^{[1]}+qf \right]\quad \text{on $[a,b]$},
\end{equation}
where in addition to the assumptions imposed on $p$, $q$, and $r$ above, the fourth coefficient $s$ is assumed to be real-valued almost everywhere and integrable on $[a,b]$ and the generalized quasi-derivative takes the form $f^{[1]}=p[f'+sf]$. Assuming the minimal operator $H_{p,q,r,s,\min}$ generated by $\tau_{p,q,r,s}$ is strictly positive, the Krein--von Neumann extension of $H_{p,q,r,s,\min}$ is characterized in \cite[Theorem 12.3]{EGNT13} by \eqref{1.11b}--\eqref{1.12b}.  A result analogous to \eqref{1.11b}--\eqref{1.12b} was shown to hold for singular three-coefficient Sturm--Liouville operators in \cite[Theorem 3.5(ii)]{FGKLNS21}, provided one replaces the values of the functions and their quasi-derivatives at the endpoints by appropriate generalized boundary values (see \cite[Theorem 2.12]{FGKLNS21} and \cite[Theorem 3.11]{GLN20}).  The Krein--von Neumann extension of the minimal operator associated with the pure differential expression $\tau_{{}_{2N}}$ of order $2N$ (where $N\in \bbN$ is fixed) acting according to
\begin{equation}\lb{1.13b}
    \tau_{{}_{2N}}f=(-1)^N\frac{d^{2N}}{dx^{2N}}f\quad \text{on $[a,b]$}
\end{equation}
was characterized by Granovskyi and Oridoroga in \cite{GO18a} (see also \cite{GO18b}).  Using an elegant argument based on Taylor polynomials, it is shown in \cite[Theorem 3.1(i)]{GO18a} that the Krein--von Neumann extension of the minimal operator $H_{2N,\min}$ corresponding to \eqref{1.13b} is characterized by the boundary conditions:
\begin{equation}\lb{1.14b}
    \begin{pmatrix}
        f(b) \\ f^{(1)}(b) \\ \vdots \\ f^{(2N-1)}(b)
        \end{pmatrix}=T_{\rm K}\begin{pmatrix}
        f(a) \\ f^{(1)}(a) \\ \vdots \\ f^{(2N-1)}(a)
        \end{pmatrix},
\end{equation}
where $T$ is the Toeplitz upper triangular matrix given by
\begin{equation}\lb{1.15b}
    T_{\rm K} = \left(\frac{(b-a)^{k-j}}{(k-j)!}\right)_{j,k=1}^{2N}.
\end{equation}
Lunyov \cite{Lu13} considered the expression $\tau_{{}_{2N}}$ in \eqref{1.13b} on the interval $[0,\infty)$ and used boundary triplet techniques to show that the Krein--von Neumann extension of the minimal operator is characterized by the boundary conditions (see \cite[Theorem 2]{Lu13})
\begin{equation}
    f^{(j)}(0)=0,\quad N\leq j\leq 2N-1.
\end{equation}
Finally, Ananieva and Budyika \cite[Proposition 5.3, part $(ii)$]{AB15} (see also \cite{AB16}) used boundary triplet techniques to characterize the Krein--von Neumann extension corresponding to the Bessel differential expression $-\tfrac{d^2}{dx^2} + \big(\nu^2-\tfrac{1}{4}\big)x^{-2}$, with the parameter $\nu \in [0,1)\backslash\{1/2\}$, on the interval $(0,\infty)$.

In this paper, we characterize by boundary conditions the Krein--von Neumann extension of the minimal operator $H_{Z,\min}$ generated by a regular Shin--Zettl quasi-differential expression $\tau_{{}_Z}$ of order $2N$ on the interval $[a,b]$ with $M\times M$ matrix-valued coefficients, assuming that the leading coefficient of $\tau_{{}_Z}$ is positive definite almost everywhere on $[a,b]$ and that $H_{Z,\min}$ is strictly positive.  Our approach is similar in spirit to \cite{CGNZ14}, \cite{EGNT13}, and \cite{FGKLNS21}, in that we also use a specially chosen basis for the kernel of $(H_{Z,\min})^*$ to formulate the boundary conditions for the Krein--von Neumann extension.  The characterization of the Friedrichs extension due to M\"oller and Zettl plays a key role in our construction.

We briefly summarize the contents of the remaining sections of this paper.  In Section \ref{s2}, we recall the basic background on Shin--Zettl quasi-differential expressions and their associated minimal and maximal operators.  In Section 3, we introduce a special basis for the kernel of $(H_{Z,\min})^*$ in Lemma \ref{L3.2} and state and prove our main result---a characterization of the Krein--von Neumann extension of $H_{Z,\min}$---in Theorem \ref{T3.4}.  We explore an equivalent characterization in Proposition \ref{p3.5}.  In Proposition \ref{P3.7}, we note the general fact that, given a densely defined strictly positive symmetric operator $S$, the Krein--von Neumann extension of $S$ and any strictly positive self-adjoint extension of $S$ are always relatively prime with respect to $S$.  Finally, in Section \ref{s4}, we consider applications of our main theorem to generalized four-coefficient regular Sturm--Liouville expressions with matrix-valued coefficients, a simple fourth-order differential expression, and the pure differential expression $\tau_{{}_{2N}}$ in \eqref{1.13b}.

\medskip
\noindent
\textbf{Notation:} If $X$ is a set and $m,n\in \bbN$, then $X^{m\times n}$ denotes the set of all $m\times n$ matrices with entries in $X$.  Thus, $G\in X^{m\times n}$ if and only if $G=(G_{j,k})_{j=1,k=1}^{m,n}$, where $G_{j,k}\in X$ for all $1\leq j\leq m$ and $1\leq k\leq n$.  In the special case when $n=1$, we will write $X^m$ instead of $X^{m\times 1}$.  For a fixed compact interval $[a,b]$ in $\bbR$, $\AC([a,b])$ denotes the set of all complex-valued functions that are absolutely continuous on $[a,b]$, $\Lp([a,b])$ denotes the set of all (equivalence classes of) Lebesgue measurable functions $f:[a,b] \to \bbC$, and $\Lp^1([a,b])$ denotes the set of all $f\in \Lp([a,b])$ such that $\int_{[a,b]} |f| < \infty$.  Here, and throughout, the integral is taken with respect to Lebesgue measure on $\bbR$, and ``a.e.''~is used as an abbreviation for the phrase ``almost everywhere with respect to Lebesgue measure.''  If $z\in \bbC$, then $\overline{z}$ denotes the complex conjugate of $z$.  If $m\in \bbN$, then $0_m$ and $I_m$ denote the zero and identity matrices, respectively, in $\bbC^{m\times m}$.  If $\mathcal{V}$ is a vector space, then $\dim(\mathcal{V})$ denotes the dimension of $\mathcal{V}$.  If $T:\mathcal{V}\to \mathcal{V}$ is a linear transformation on the vector space $\mathcal{V}$, then $\ker(T)$ denotes the kernel (i.e., null space) of $T$.  Finally, $\bbZ_{\geq 0}$ denotes the set of nonnegative integers, and ``$:=$'' means ``is defined to be equal to.''

\section{Even order Regular Quasi-Differential Operators} \lb{s2}
 
We begin by recalling several basic facts on even order regular quasi-differential operators.  This material may be found in many sources; we refer to \cite[Sections I.2, II, \& IV]{EM99}, \cite{EZ79}, \cite{GMP13}, \cite{MZ95a}, \cite{Sh38}, and \cite{Ze75} for detailed treatments, including proofs.  In fact, \cite{MZ95a} and \cite{MZ95b} contain all of the background required here.  To construct an even order regular quasi-differential expression of Shin--Zettl type, we introduce the following hypothesis which is assumed throughout Sections \ref{s2} and \ref{s3}.
 
\begin{hypothesis}\lb{h2.1}
$M,N\in \bbN$ are fixed, $[a,b]$ is a compact subinterval in $\bbR$, $W\in \Lp^1([a,b])^{M\times M}$ is positive definite a.e.~on $[a,b]$, and
\begin{equation*}
Z=(Z_{j,k})_{j,k=1}^{2N}\in \big[\Lp^1([a,b])^{M\times M}\big]^{2N\times 2N}
\end{equation*}
satisfies the following conditions:
\medskip
\begin{itemize}
\item[(A1)]  $Z_{j,j+1}$ is invertible a.e.~on $[a,b]$ for $1\leq j\leq 2N-1$
\medskip
\item[(A2)]  $Z_{j,k}=0_M$ a.e.~on $[a,b]$ for $2\leq j+1<k\leq 2N$
\medskip
\item[(A3)]  $Z = J_{M,2N}Z^*J_{M,2N}$, where
\begin{equation}\lb{2.1}
J_{M,2N}=\big((-1)^j\delta_{j,2N+1-k} I_M\big)_{j,k=1}^{2N}\in \big[\bbC^{M\times M}\big]^{2N\times 2N}.
\end{equation}
\end{itemize}
\end{hypothesis}
Assuming Hypothesis \ref{h2.1}, the quasi-derivatives generated by $Z$ are defined as follows.  Set
\begin{equation}\lb{2.2}
y^{[0]}_Z:=y,\quad y\in \mathfrak{D}^{[0]}_Z([a,b]):=\Lp([a,b])^M,
\end{equation}
and define $y_Z^{[j]}$ for $1\leq j\leq 2N$ inductively by
\begin{equation}\lb{2.3}
\begin{split}
y^{[j]}_Z&:=Z_{j,j+1}^{-1}\Bigg[\big(y^{[j-1]}_Z\big)'-\sum_{k=1}^j Z_{j,k}y^{[k-1]}_Z\Bigg],\\
y\in \mathfrak{D}_Z^{[j]}([a,b])&:=\left\{g\in\mathfrak{D}_Z^{[j-1]}([a,b])\,\big|\, g^{[j-1]}_Z\in \text{AC}([a,b])^M\right\},
\end{split}
\end{equation}
where $Z_{2N,2N+1}:=I_M$ a.e.~on $[a,b]$ and the prime denotes differentiation with respect to the independent variable on $[a,b]$. 
 
\smallskip
\noindent
{\bf Notational convention:} Since $Z$ is fixed, for ease of notation, we shall, from this point on, simply write $y^{[j]}$ for the $j$th quasi-derivative of $y$, instead of $y_Z^{[j]}$.
\medskip
 
The quasi-differential expression $\tau_{{}_Z}$ generated by $Z$ is defined by
\begin{equation}\lb{2.4}
\tau_{{}_Z} y:=(-1)^NW^{-1}y^{[2N]},\quad y\in \mathfrak{D}_Z^{[2N]}([a,b]),
\end{equation}
and $Z_{N,N+1}$ is called the {\it leading coefficient} of $\tau_{{}_Z}$.

Introducing the Lagrange bracket $[\dott,\dott]_Z$ by
\begin{equation}\lb{2.5}
[f,g]_Z=(-1)^N\sum_{j=0}^{2N-1}(-1)^{1-j}{\left(g^{[2N-j-1]}\right)}^*f^{[j]},\quad f,g\in \mathfrak{D}_Z^{[2N]}([a,b]),
\end{equation}
where ${}^*$ denotes the Hermitian transpose of a matrix, it follows that $[f,g]_Z\in \AC([a,b])^M$ for all $f,g\in \mathfrak{D}_Z^{[2N]}([a,b])$ and the Lagrange identity holds in the following form.
 
\begin{lemma}[Lagrange identity, {\cite[Lemma 3.3]{MZ95a}}]\lb{l2.2}
Assume Hypothesis \ref{h2.1}.  If $f,g\in \mathfrak{D}_Z^{[2N]}([a,b])$, then
\begin{equation}\lb{2.6}
g^*W(\tau_{{}_Z}f) - (\tau_{{}_Z}g)^*Wf = [f,g]_Z'\,\text{ a.e.~on $[a,b]$}.
\end{equation}
\end{lemma}
For later use, it is convenient to extend the definition of the Lagrange bracket to matrix-valued functions of the form $F:[a,b]\to\bbC^{M\times M}$ whose columns $F_1,\ldots,F_M$ are in $\mathfrak{D}_Z^{[2N]}([a,b])$.  For such $F$, we define
\begin{equation}\lb{2.7a}
        F^{[\ell]}:=\Big(\,
            F_1^{[\ell]}\,\Big|\,F_2^{[\ell]}\,\Big|\,\cdots\,\Big|\, F_M^{[\ell]}\,\Big),\quad 0\leq \ell\leq 2N,
\end{equation}
and
\begin{equation}\lb{2.7b}
    \tau_{{}_Z}F:=(-1)^NW^{-1}F^{[2N]}.
\end{equation}
The definitions in \eqref{2.7a} and \eqref{2.7b} imply $F^{[\ell]}$, $0\leq \ell\leq 2N$, is an $M\times M$ matrix-valued function and
\begin{equation}\lb{2.9c}
    \tau_{{}_Z}F=\big(\,\tau_{{}_Z}F_1\,\big|\,\tau_{{}_Z}F_2\,\big|\,\cdots\,\big|\,\tau_{{}_Z}F_M\,\big).
\end{equation}
If $F,G:[a,b]\to\bbC^{M\times M}$ are functions whose columns $F_1,\ldots,F_M$ and $G_1,\ldots,G_M$, respectively, are in $\mathfrak{D}_Z^{[2N]}([a,b])$, then we define their Lagrange bracket by
\begin{equation}\lb{2.7c}
    [F,G]_Z:=(-1)^N\sum_{j=0}^{2N-1}(-1)^{1-j}\left(G^{[2N-j-1]}\right)^*F^{[j]},
\end{equation}
and obtain a Lagrange identity similar to Lemma \ref{l2.2}.
\begin{lemma}\lb{l2.3}
    Assume Hypothesis \ref{h2.1}.  If $F,G:[a,b] \to \bbC^{M\times M}$ are functions whose columns $F_1,\ldots,F_M$ and $G_1,\ldots,G_M$, respectively, are in $\mathcal{D}_Z^{[2N]}([a,b])$, then
    \begin{equation}
      G^*W(\tau_{{}_Z}F)-\left(\tau_{{}_Z}G\right)^*WF=[F,G]_Z'.
    \end{equation}
\end{lemma}
\begin{proof}
    By \eqref{2.7c},
    \begin{align}
        [F,G]_Z
            &\no =(-1)^N\sum_{\ell=0}^{2N-1}(-1)^{1-\ell}\left(\left(G_j^{[2N-\ell-1]}\right)^*F_k^{[\ell]}\right)_{j,k=1}^M\\
            &\no =\left((-1)^N\sum_{\ell=0}^{2N-1}(-1)^{1-\ell}\left(G_j^{[2N-\ell-1]}\right)^*F_k^{[\ell]}\right)_{j,k=1}^M\\
            &=\big([F_k,G_j]_Z\big)_{j,k=1}^M. \lb{2.11a}
    \end{align}
    Therefore, taking the derivative componentwise, we compute:
    \begin{align}
        [F,G]_Z'&\no =\left([F_k,G_j]_Z'\right)_{j,k=1}^M\\
            &\no =\left(G_j^*W(\tau_{{}_Z} F_k)-\left(\tau_{{}_Z} G_j\right)^*WF_k\right)_{j,k=1}^M\\
            &=G^*W(\tau_{{}_Z} F) - (\tau_{{}_Z} G)^*WF\lb{2.11b}
    \end{align}
    by linearity, Lemma \ref{l2.2}, \eqref{2.7b}, and \eqref{2.9c}.
\end{proof}
\medskip
In order to define the maximal and minimal operators associated to $\tau_{{}_Z}$, we introduce the Hilbert space $\Lp^2_W([a,b])$ of all (equivalence classes of) $f\in \Lp([a,b])^M$ for which $f^*Wf\in \Lp^1([a,b])$ equipped with the inner product
\begin{equation}\lb{2.7}
\langle f,g\rangle_W:= \int_{[a,b]} g^*Wf,\quad f,g\in \Lp^2_W([a,b]).
\end{equation}
We shall denote the identity operator on $\Lp^2_W([a,b])$ by $I_W$.
 
The {\it maximal operator} $H_{Z,\max}$ associated to $\tau_{{}_Z}$ is defined by
\begin{align}
&H_{Z,\max}f = \tau_{{}_Z}f,\lb{2.8}\\
&f\in \dom(H_{Z,\max})=\big\{y\in \Lp_W^2([a,b])\,\big|\, y\in \mathfrak{D}_Z^{[2N]}([a,b]),\, \tau_{{}_Z}y\in \Lp_W^2([a,b])\big\},\no
\end{align}
and the {\it minimal operator} $H_{Z,\min}$ associated to $\tau_{{}_Z}$ is defined by
\begin{align}\lb{2.9}
&H_{Z,\min}f = \tau_{{}_Z}f,\\
&f\in \dom(H_{Z,\min})=\big\{y\in \dom(H_{Z,\max})\,\big|\, y^{[j-1]}(a)=y^{[j-1]}(b)=0,\, 1\leq j\leq 2N\big\}.\no
\end{align}
One can show (see, e.g., \cite[Theorem 4.2]{MZ95a}) that $H_{Z,\max}$ and $H_{Z,\min}$ are densely defined and satisfy the following adjoint relations:
\begin{equation}\lb{2.17}
(H_{Z,\min})^*=H_{Z,\max}\quad \text{and}\quad (H_{Z,\max})^*=H_{Z,\min}.
\end{equation}
The equalities in \eqref{2.17} imply that $H_{Z,\max}$ and $H_{Z,\min}$ are closed.  Moreover, an elementary calculation using the Lagrange identity and the boundary conditions for functions in $\dom(H_{Z,\min})$ reveals that $H_{Z,\min}$ is symmetric.  Since $\tau_{{}_Z}$ is regular on $[a,b]$, the deficiency indices of $H_{Z,\min}$ satisfy (see \cite[Equation (2.2)]{MZ95b})
\begin{equation}\lb{2.11}
\dim(\ker((H_{Z,\min})^*\pm iI_W)) = \dim(\ker(H_{Z,\max}\pm iI_W)) = 2MN.
\end{equation}
Hence, $H_{Z,\min}$ has a self-adjoint extension. If $H$ is a self-adjoint extension of $\Hz{\min}$, then \eqref{2.17} implies
\begin{equation}
    \Hz{\min} \subseteq H \subseteq \Hz{\max},
\end{equation}
so $H$ is a self-adjoint restriction of $\Hz{\max}$. Thus, the action of a self-adjoint extension of $\Hz{\min}$ coincides with the action of $\Hz{\max}$. As a consequence, a self-adjoint extension of $\Hz{\min}$ is determined uniquely by its domain.
 
For $y\in\dom(H_{Z,\max})$, we introduce the notation:
              \begin{equation}\lb{2.20Y}
              Y(x)=\left(\begin{array}{c}
                           y^{[0]}(x)\\
                           y^{[1]}(x)\\
                           \vdots\\
                           y^{[2N-1]}(x)
              \end{array}\right),\quad x\in\{a,b\}.
              \end{equation}
The following theorem, which is a special case of \cite[Theorem 2.4]{MZ95b}, permits one to construct self-adjoint extensions of $H_{Z,\min}$ by imposing boundary conditions at the endpoints of $[a,b]$.
\begin{theorem}[{\cite[Theorem 2.4]{MZ95b}}]\lb{T2.3}
              Assume Hypothesis \ref{h2.1} and suppose that $A,B\in \big[\bbC^{M\times M}\big]^{2N\times 2N}$.  The operator $H_{Z,A,B}$ defined by
              \begin{equation}\lb{2.13}
                           \begin{split}
                                         &H_{Z,A,B}f=H_{Z,\max}f,\\
                                 &f\in\dom(H_{Z,A,B})=\{y\in\dom(H_{Z,\max})\,|\, AY(a)=BY(b)\},
                           \end{split}
              \end{equation}
is a self-adjoint extension of $H_{Z,\min}$ if and only if
              \begin{equation}\lb{2.14}
                            \rank(A\,|\,B)=2MN\quad\text{and}\quad AJ_{M,2N}A^*=BJ_{M,2N}B^*,
              \end{equation}
where $(A\,|\,B)$ is viewed as an element of $\bbC^{2MN\times 4MN}$.
\end{theorem}

Recall that $\Hz{\min}$ is said to be {\it bounded from below} if there exists $\kappa \in \bbR$ such that
\begin{equation}
    \langle f, \Hz{\min} f\rangle_W \geq
        \kappa \langle f, f\rangle_W, \quad f \in \dom(\Hz{\min}).
\end{equation}
Define the map $\Gamma: \dom(H_{Z,\max})\to (\bbC^M)^{2N}$ by
\begin{equation}
\Gamma u =\begin{pmatrix}
              u^{[0]}(a)\\
              u^{[1]}(a)\\
              \vdots\\
              u^{[N-1]}(a)\\
              u^{[0]}(b)\\
              u^{[1]}(b)\\
              \vdots\\
              u^{[N-1]}(b)
\end{pmatrix}, \quad u\in \dom(H_{Z,\max}). \lb{3.3}
\end{equation}
The reason for introducing the map $\Gamma$ is that $H_{Z,\min}$ is bounded from below when $Z_{N,N+1}$ is positive definite a.e.~on $[a,b]$, and the domain of its Friedrichs extension coincides with $\ker(\Gamma)$. This result is due to M\"{o}ller and Zettl \cite{MZ95a}.
\begin{theorem}[{\cite[Theorem 8.1]{MZ95a}}]\lb{t2.5}
Assume Hypothesis \ref{h2.1}.  If $Z_{N,N+1}$ is positive definite a.e.~on $[a,b]$, then $H_{Z,\min}$ is bounded from below, and the domain of the Friedrichs extension $H_{Z,\rm F}$ of $H_{Z,\min}$ is
\begin{equation}\lb{3.4}
              \begin{split}
                           \dom(H_{Z,\rm F})&=\ker(\Gamma)\\
                           &=\big\{y\in \dom(H_{Z,\max})\,\big|\, y^{[j-1]}(a)=y^{[j-1]}(b)=0,\, 1\leq j\leq N\big\}.
              \end{split}
\end{equation}
\end{theorem}
The characterization of $H_{Z,\rm F}$ given in Theorem \ref{t2.5} will play an important role in our characterization of the Krein--von Neumann extension in the next section.
 
\section{Main Results} \lb{s3}

In this section, we assume in addition to Hypothesis \ref{h2.1} that $Z_{N,N+1}$ is positive definite a.e.~on $[a,b]$ and that the minimal operator $H_{Z,\min}$ associated to $\tau_{{}_Z}$ is strictly positive:
\begin{equation}
\Hz{\min} \geq \varepsilon I_W \, \text{ for some $\varepsilon \in (0,\infty)$}.\lb{3.1}
\end{equation}
As a consequence of \eqref{3.1},
\begin{equation}\lb{3.2}
\begin{split}
\dim\big(\ker\big((H_{Z,\min})^*-zI_W\big)\big) = \dim\big(\ker(H_{Z,\max}-zI_W)\big) = 2MN,&\\
z\in \bbC\backslash [\varepsilon,\infty).&
\end{split}
\end{equation}
In particular, \eqref{3.2} implies that $\ker(H_{Z,\max})$ is a $2MN$-dimensional subspace of the Hilbert space $\Lp_W^2([a,b])$.  We shall characterize by boundary conditions the Krein--von Neumann extension $H_{Z,\rm K}$ of $H_{Z,\min}$ in terms of a specially chosen basis for $\ker(H_{Z,\max})$.  This basis is characterized in the following lemma.

\begin{lemma}\lb{L3.2}
              Assume Hypothesis \ref{h2.1} and suppose that $Z_{N,N+1}$ is positive definite a.e.~on $[a,b]$.  If \eqref{3.1} holds, then there exists a unique basis $\{\varphi_{j,k}\}_{j=1, k=1}^{2N, M}$ of $\ker(H_{Z,\max})$ such that
\begin{equation}\lb{3.5}
\Gamma\varphi_{j,k}=(\delta_{j,\ell}e_k)_{\ell=1}^{2N},
\end{equation}
where $\{e_k\}_{k=1}^M$ denotes the standard basis of $\bbC^M$.
\end{lemma}
\begin{proof}
It suffices to show that  $\Gamma\big|_{\ker(H_{Z,\max})}$ is a bijection, for one may then verify that $\{\varphi_{j,k}\}_{j=1, k=1}^{2N, M}$ defined by
 
\begin{equation}\lb{3.6}
\varphi_{j,k} = \left(\Gamma\big|_{\ker(H_{Z,\max})}\right)^{-1} (\delta_{j,\ell}e_k)_{\ell=1}^{2N},\quad 1\leq j\leq 2N,\, 1\leq k\leq M,
\end{equation}
is a basis for $\ker(\Hz{\max})$ that fulfills \eqref{3.5}. To prove injectivity, suppose $y\in \ker(H_{Z,\max})$ and $\Gamma y=0$. By equation \eqref{3.4}, $y\in \dom(H_{Z,\rm F})$; thus, $H_{Z, \rm F} y = H_{Z, \max} y = 0$. Moreover, \eqref{3.1} implies $\Hz{\rm F} \geq \varepsilon I_W$ (cf.~Theorem \ref{T1.1}), so that
              \begin{equation}
       \varepsilon\langle y,y\rangle_W\leq\langle y,H_{Z, \rm F}y \rangle_W=0.
              \end{equation}
Thus, $y=0$. Therefore, $\Gamma\big|_{\ker(H_{Z,\max})}$ has a trivial kernel and is thus injective. Since
              \begin{equation}
              \dim(\ker(H_{Z,\max}))=2MN=\dim\big((\bbC^M)^{2N}\big),
              \end{equation}
$\Gamma\big|_{\ker(H_{Z,\max})}$ is also surjective. Therefore, we retrieve a basis $\{\varphi_{j,k}\}_{j=1, k=1}^{2N, M}$ by  \eqref{3.6}, and by injectivity of $\Gamma\big|_{\ker(H_{Z,\max})}$ we also conclude that it is the unique basis for $\ker(H_{Z,\max})$ that satisfies \eqref{3.5}.
\end{proof}
\begin{remark}\lb{r3.3}
The definition of $\Gamma$ in \eqref{3.3} implies for $1\leq j \leq 2N$ and $1 \leq k \leq M$,
\begin{equation}
\big( \Gamma \varphi_{j,k}\big)_{\ell}= \begin{cases}
    \varphi_{j,k}^{[\ell-1]}(a), & 1 \leq \ell \leq N,
    \\[2mm]  \varphi_{j,k}^{[\ell-N-1]}(b), & N+1 \leq \ell \leq 2N,
\end{cases}
\end{equation}
where the subscript $\ell$ on the left-hand side denotes the $\ell$th component of $\Gamma\varphi_{j,k}$.  Therefore, \eqref{3.5} yields for $1\leq j \leq 2N$, $1\leq k \leq M$, and $1 \leq \ell \leq N$,
\begin{align}
    \varphi_{j,k}^{[\ell-1]}(a)&=\delta_{j,\ell}e_k \quad \text{and} \quad
    \varphi_{j,k}^{[\ell-1]}(b)=\delta_{j, \ell+N}e_k.\lb{3.10}
\end{align}
\hfill $\diamond$
\end{remark}
Using the basis $\{\varphi_{j,k}\}_{j=1,k=1}^{2N,M}$ for $\ker(H_{Z,\max})$ prescribed in Lemma \ref{L3.2}, we will characterize the Krein--von Neumann extension $\Hz{\rm K}$ of $H_{Z,\min}$ by boundary conditions.  In fact, we shall prove that $H_{Z,\rm K}$ is of the form \eqref{2.13}.  That is, we shall show that every function $y\in \dom(H_{Z,\rm K})$ satisfies boundary conditions of the form
\begin{equation}
A_{\rm K}Y(a) = B_{\rm K}Y(b)
\end{equation}
for a pair of fixed ($y$-independent) matrices $A_{\rm K}, B_{\rm K}\in \big[\bbC^{M\times M}\big]^{2N\times 2N}$ which satisfy \eqref{2.14}.
 
To determine $A_{\rm K}$ and $B_{\rm K}$, we recall $\eqref{1.8}$, which now takes the form
\begin{equation}\lb{3.12}
    \dom(H_{Z, \rm K})=\dom(H_{Z,\min})\dotplus\ker{(H_{Z, \max})}.
\end{equation}
If $y\in \dom(H_{Z, \rm K})$, then by \eqref{3.12} there exist scalars $\{c_{j,k}\}_{j=1,k=1}^{2N,M} \subset \bbC$ and some $\psi\in \dom(H_{Z,\min})$ such that
\begin{equation}\lb{3.13}
    y=\psi+\sum_{j=1,k=1}^{2N,M}c_{j,k}\varphi_{j,k}.
\end{equation}
Therefore, since \eqref{2.9} implies
\begin{equation}\lb{3.14b}
    \psi^{[\ell-1]}(x) = 0, \quad 1 \leq \ell \leq 2N, \, x \in \{a,b\},
\end{equation}
we have
\begin{equation}\lb{3.14}
    y^{[\ell-1]}(x)=\sum_{j=1, k=1}^{2N,M} c_{j,k}\varphi_{j,k}^{[\ell-1]}(x),\quad 1\leq \ell \leq N,\, x\in \{a,b\}.
\end{equation}
Letting $x=a$ in \eqref{3.14} and applying \eqref{3.10}, we obtain for $1\leq \ell\leq N$:
\begin{equation}\lb{3.15}
    y^{[\ell-1]}(a)=\sum_{j=1,k=1}^{2N,M} c_{j,k}\varphi_{j,k}^{[\ell-1]}(a)=\sum_{j=1,k=1}^{2N,M} c_{j, k}\delta_{j,\ell}e_k=\sum_{k=1}^M c_{\ell, k}e_k.
\end{equation}
Therefore, taking the $k$th component throughout \eqref{3.15} yields
\begin{equation}\lb{3.16}
    c_{j,k}=\big(y^{[j-1]}(a)\big)_{k},\quad 1\leq j\leq N,\, 1\leq k\leq M,
\end{equation}
where the subscript $k$ on the right-hand side denotes the $k$th component of a vector in $\bbC^M$. Similarly, letting $x=b$ in \eqref{3.14} and applying \eqref{3.10}, we obtain for $1\leq \ell\leq N$:
\begin{equation}\lb{3.17}
    y^{[\ell-1]}(b)=\sum_{j=1,k=1}^{2N,M} c_{j,k}\varphi_{j,k}^{[\ell-1]}(b)
    =\sum_{j=1,k=1}^{2N,M} c_{j, k}\delta_{j,\ell+N}e_k
    =\sum_{k=1}^Mc_{\ell+N, k}e_k.
\end{equation}
Therefore, taking the $k$th component throughout \eqref{3.17} yields
\begin{equation}\lb{3.18}
c_{j,k}=\big(y^{[j-1-N]}(b)\big)_{k},\quad N+1\leq j\leq 2N,\, 1\leq k\leq M.
\end{equation}
Using \eqref{3.16} and \eqref{3.18}, \eqref{3.13} can be recast as
\begin{equation}\lb{3.19}
    y=\psi+\sum_{j=1,k=1}^{N,M}\big(y^{[j-1]}(a)\big)_{k}\varphi_{j,k}
    +\sum_{j=1,k=1}^{N,M}\big(y^{[j-1]}(b)\big)_{k}\varphi_{j+N,k}.
\end{equation}
In particular, \eqref{3.14b} and \eqref{3.19} imply
\begin{equation}\lb{3.20}
    \begin{split}
    y^{[\ell-1]}(x)=\sum_{j=1,k=1}^{N,M}\big(y^{[j-1]}(a)\big)_{k}\varphi_{j,k}^{[\ell-1]}(x)
    +\sum_{j=1,k=1}^{N,M}\big(y^{[j-1]}(b)\big)_{k}\varphi_{j+N,k}^{[\ell-1]}(x),&\\
    N+1\leq \ell\leq 2N,\, x\in \{a,b\}.&
    \end{split}
\end{equation}
Letting $x=a$ in \eqref{3.20} and rearranging, we obtain
\begin{equation}\lb{3.21}
    \begin{split}
    -\sum_{j=1,k=1}^{N,M}\big(y^{[j-1]}(a)\big)_k\varphi_{j,k}^{[\ell-1]}(a)+y^{[\ell-1]}(a)=\sum_{j=1,k=1}^{N,M}\big(y^{[j-1]}(b)\big)_k\varphi_{j+N,k}^{[\ell-1]}(a),&\\
    N+1\leq \ell\leq 2N.&
    \end{split}
\end{equation}
We introduce the following notation:
\begin{equation}\lb{3.22}
    \pm\varphi_j^{[\ell-1]}=\left( \begin{array}{c|c|c|c}
      \pm\varphi_{j,1}^{[\ell-1]} & \pm\varphi_{j,2}^{[\ell-1]} & \cdots & \pm\varphi_{j,M}^{[\ell-1]}
    \end{array}\right), \quad 1\leq j,\ell \leq 2N.
\end{equation}
We view $\pm\varphi_j^{[\ell-1]}$ as $M \times M$ matrices, instead of row vectors of column vectors, and remark that, by \eqref{3.10}, for $1\leq \ell \leq N$, $1\leq j \leq 2N$,
\begin{equation}\lb{3.9}
    \pm\varphi_j^{[\ell-1]}(a)=\pm\delta_{j,\ell}I_M
    \quad \text{and}\quad
    \pm\varphi_{j}^{[\ell-1]}(b)=\pm\delta_{j-N,\ell}I_M.
\end{equation}
Using the notation in \eqref{3.22}, equation \eqref{3.21} can be rewritten as
\begin{equation}\lb{3.23}
    \begin{split}
    -\sum_{j=1}^{N}\varphi_{j}^{[\ell-1]}(a)y^{[j-1]}(a)+y^{[\ell-1]}(a)=\sum_{j=1}^{N}\varphi_{j+N}^{[\ell-1]}(a)y^{[j-1]}(b),&\\
    N+1\leq \ell\leq 2N.&
    \end{split}
\end{equation}
In turn, equation \eqref{3.23} may be recast in terms of matrix products as follows (cf.~\eqref{2.20Y}):
\begin{align}\lb{3.24}
    &\left(\begin{array}{c|c|c|c|c|c|c|c}
     -\varphi_1^{[\ell-1]}(a) & \cdots & -\varphi_{N}^{[\ell-1]}(a) & 0_M & \cdots & I_M & \cdots & 0_M
\end{array}\right)Y(a)
\\
&\quad =\left(\begin{array}{c|c|c|c|c|c|c|c}
     \varphi_{N+1}^{[\ell-1]}(a) & \cdots & \varphi_{2N}^{[\ell-1]}(a) & 0_M  & \cdots & 0_M & \cdots & 0_M
\end{array}\right)Y(b),\no
\end{align}
where the $I_M$ on the left-hand side and the second $0_M$ on the right-hand side are positioned in the $\ell$th columns.
Similarly, letting $x=b$ in \eqref{3.20}, one obtains
\begin{align}\lb{3.25}
    &\left(\begin{array}{c|c|c|c|c|c|c|c}
     \varphi_1^{[\ell-1]}(b) & \cdots & \varphi_{N}^{[\ell-1]}(b) & 0_M & \cdots & 0_M & \cdots & 0_M
\end{array}\right)Y(a)
\\
&\quad =\left(\begin{array}{c|c|c|c|c|c|c|c}
     -\varphi_{N+1}^{[\ell-1]}(b) & \cdots & -\varphi_{2N}^{[\ell-1]}(b) & 0_M  & \cdots & I_M & \cdots & 0_M
\end{array}\right) Y(b),\no
\end{align}
where the second $0_M$ on the left-hand side and the $I_M$ on the right-hand side are positioned in the $\ell$th columns.
We may stack \eqref{3.24} and \eqref{3.25} for $N+1 \leq \ell \leq 2N$ into a single matrix equation
\begin{equation}\lb{3.27}
    A_{\rm K}Y(a)=B_{\rm K}Y(b),
\end{equation}
where $A_{\rm K}$ and $B_{\rm K}$, which are $2N\times 2N$ block matrices with $M\times M$ block components, are defined by
\begin{align}
    A_{\rm K}&=\left(\begin{array}{c c c c | c c c c}
        -\varphi_1^{[N]}(a) & -\varphi_2^{[N]}(a) & \cdots & -\varphi_{N}^{[N]}(a) & I_M & 0_M & \cdots & 0_M
        \\[1mm]
        -\varphi_1^{[N+1]}(a) & -\varphi_2^{[N+1]}(a) & \cdots & -\varphi_{N}^{[N+1]}(a) & 0_M & I_M & \cdots & 0_M
        \\
        \vdots & \vdots & \ddots & \vdots & \vdots & \vdots & \ddots& \vdots
        \\
        -\varphi_1^{[2N-1]}(a) & -\varphi_2^{[2N-1]}(a) & \cdots & -\varphi_{N}^{[2N-1]}(a) & 0_M & 0_M & \cdots & I_M
        \\[1mm]
        \hline\\[-3.5mm]
        \varphi_1^{[N]}(b) & \varphi_2^{[N]}(b) & \cdots & \varphi_{N}^{[N]}(b) & 0_M & 0_M & \cdots & 0_M
        \\[1mm]
        \varphi_1^{[N+1]}(b) &
        \varphi_2^{[N+1]}(b) & \cdots & \varphi_{N}^{[N+1]}(b) & 0_M & 0_M & \cdots & 0_M
        \\
        \vdots & \vdots & \ddots & \vdots & \vdots & \vdots & \ddots & \vdots
        \\
        \varphi_1^{[2N-1]}(b) &
        \varphi_2^{[2N-1]}(b) & \cdots & \varphi_{N}^{[2N-1]}(b) & 0_M & 0_M & \cdots & 0_M
    \end{array}\right),\no\\[3mm]
    B_{\rm K} &= \left(
        \begin{array}{c c c c | c c c c}
        \varphi_{N+1}^{[N]}(a) & \varphi_{N+2}^{[N]}(a) & \cdots & \varphi_{2N}^{[N]}(a) & 0_M & 0_M & \cdots & 0_M
        \\[1mm]
        \varphi_{N+1}^{[N+1]}(a) & \varphi_{N+2}^{[N+1]}(a) & \cdots & \varphi_{2N}^{[N+1]}(a) & 0_M & 0_M & \cdots & 0_M
        \\
        \vdots & \vdots & \ddots & \vdots & \vdots & \vdots & \ddots & \vdots
        \\
        \varphi_{N+1}^{[2N-1]}(a) & \varphi_{N+2}^{[2N-1]}(a) & \cdots & \varphi_{2N}^{[2N-1]}(a) & 0_M & 0_M & \cdots & 0_M
        \\[1mm]
        \hline\\[-3.5mm]
        -\varphi_{N+1}^{[N]}(b) & -\varphi_{N+2}^{[N]}(b) & \cdots & -\varphi_{2N}^{[N]}(b) & I_M & 0_M & \cdots & 0_M
        \\[1mm]
        -\varphi_{N+1}^{[N+1]}(b) &
        -\varphi_{N+2}^{[N+1]}(b) & \cdots & -\varphi_{2N}^{[N+1]}(b) & 0_M & I_M & \cdots & 0_M
        \\
        \vdots & \vdots & \ddots & \vdots & \vdots & \vdots & \ddots & \vdots
        \\
        -\varphi_{N+1}^{[2N-1]}(b) &
        -\varphi_{N+2}^{[2N-1]}(b) & \cdots & -\varphi_{2N}^{[2N-1]}(b) & 0_M & 0_M & \cdots & I_M
    \end{array} \right).\no
\end{align}
 
By introducing
\begin{equation}\lb{3.29}
    \Phi_X = \Big(\varphi_{k+X}^{[j+N-1]}\Big)_{j,k=1}^N, \quad X \in \{0,N\},
\end{equation}
the matrices $A_{\rm K}$ and $B_{\rm K}$ may be written in block form as
\begin{equation}\lb{3.30}
    A_{\rm K}=\left( \begin{array}{c|c}
        -\Phi_0(a) & I_{M,N} \\ \hline
        \Phi_0(b) & 0_{M,N}
    \end{array}\right)\quad\text{and}\quad
    B_{\rm K}=\left( \begin{array}{c|c}
         \Phi_N(a) & 0_{M,N} \\ \hline
        -\Phi_N(b) & I_{M,N}
    \end{array}\right),
\end{equation}
where $I_{M,N} = (\delta_{j,k}I_M)_{j,k=1}^N$ and $0_{M,N}$ is the $N \times N$ matrix in which all entries are $0_M$.  The main result of this paper may be stated as follows.

\begin{theorem}\lb{T3.4}
Assume Hypothesis \ref{h2.1} and suppose that $Z_{N,N+1}$ is positive definite a.e.~on $[a,b]$. If \eqref{3.1} holds and $\{\varphi_{j,k}\}_{j=1,k=1}^{2N,M}$ is the basis for $\ker(H_{Z,\max})$ defined by \eqref{3.5}, then the domain of the Krein--von Neumann extension $H_{Z,{\rm K}}$ of $H_{Z,\min}$ is given by
\begin{equation} \lb{3.31}
    \dom(H_{Z,{\rm K}}) = \{y\in \dom(H_{Z,\max})\,|\, A_{\rm K}Y(a) = B_{\rm K}Y(b)\},
\end{equation}
where $A_{\rm K}$ and $B_{\rm K}$ are defined by \eqref{3.29} and \eqref{3.30}.
\end{theorem}
\begin{proof}
The arguments in equations \eqref{3.12}--\eqref{3.30} imply
\begin{equation}\lb{3.32}
     \dom(H_{Z,{\rm K}}) \subseteq \{y\in \dom(H_{Z,\max})\,|\, A_{\rm K}Y(a) = B_{\rm K}Y(b)\}.
\end{equation}
Since the self-adjoint operator $\Hz{\rm K}$ does not have a proper self-adjoint extension, in order to establish \eqref{3.31}, it suffices to show that the set on the right-hand side of \eqref{3.32} is the domain of a self-adjoint extension of $H_{Z,\min}$. In turn, by Theorem \ref{T2.3}, it suffices to show that  $A_{\rm K}$ and $B_{\rm K}$ satisfy \eqref{2.14}.
 
In view of the $I_{M,N}$ identity blocks in \eqref{3.30}, it is clear that $\rank(A_{\rm K} \, | \, B_{\rm K}) = 2MN$. Thus, it remains to show
\begin{equation}\lb{3.33}
    A_{\rm K}J_{M,2N}A_{\rm K}^*=B_{\rm K}J_{M,2N}B_{\rm K}^*.
\end{equation}
Writing $J_{M,2N}$ in block form as in \eqref{3.30},
\begin{equation}\lb{3.34}
    J_{M,2N}=\left(\begin{array}{c|c}
        0_{M,N} & J_{M,N}\\[1mm]
        \hline\\[-3.5mm]
        (-1)^NJ_{M,N} &0_{M,N}\\
    \end{array}\right),
\end{equation}
where
\begin{equation}\lb{3.34a}
J_{M,N}=\big((-1)^j\delta_{j,N+1-k} I_M\big)_{j,k=1}^{N}\in \big[\bbC^{M\times M}\big]^{N\times N},
\end{equation}
we obtain
\begin{align}
    A_{\rm K}J_{M,2N}A_{\rm K}^*&=\left(\begin{array}{c|c}
         (-1)^{N+1}J_{M,N}\Phi_0(a)^*-\Phi_0(a)J_{M,N}& (-1)^N J_{M,N}\Phi_0(b)^* \\[1mm]
         \hline\\[-3.5mm]
         \Phi_0(b)J_{M,N}& 0_{M,N}
    \end{array}\right), \lb{3.35}\\
    B_{\rm K}J_{M,2N}B_{\rm K}^*&=\left(\begin{array}{c|c}
         0_{M,N}& \Phi_N(a)J_{M,N} \\[1mm]
         \hline\\[-3.5mm]
         (-1)^NJ_{M,N}\Phi_N(a)^* & (-1)^{N+1}J_{M,N}\Phi_{N}(b)^*-\Phi_N(b)J_{M,N}
    \end{array}\right).\lb{3.36}\end{align}
In order to prove \eqref{3.33}, it suffices to prove equality between each of the respective block components on the right-hand sides in \eqref{3.35} and \eqref{3.36}. We show the equalities for the $(1,1)$ and $(1,2)$ block components; that is, we prove
\begin{align}
     0_{M,N}&=(-1)^{N+1}J_{M,N}\Phi_0(a)^*-\Phi_0(a)J_{M,N},\lb{3.37}\\
    \Phi_N(a)J_{M,N}&=(-1)^N J_{M,N}\Phi_0(b)^*.\lb{3.38}
\end{align}
The remaining two equalities for the $(2,1)$ and $(2,2)$ block components can be shown in an entirely analogous manner.  Using the elementary relation $J_{M,N}^{-1}=(-1)^{N+1}J_{M,N}$, we isolate $\Phi_0(a)$ and $\Phi_N(a)$ in \eqref{3.37} and \eqref{3.38}, respectively, and find that \eqref{3.37} and \eqref{3.38} are equivalent to
\begin{align}
   \Phi_0(a)&=J_{M,N}\Phi_0(a)^*J_{M,N},\lb{3.39}\\
    \Phi_N(a) &= -J_{M,N}\Phi_0(b)^*J_{M,N}.\lb{3.40}
\end{align}
Entrywise, \eqref{3.39} and \eqref{3.40} reduce to showing that, for $1 \leq j,k \leq N$,
\begin{align}
    \varphi_k^{[N+j-1]}(a)&=(-1)^{N+j+k+1}\left(\varphi_{N+1-j}^{[2N-k]}(a)\right)^*,\lb{3.41}\\
    \varphi_{N+k}^{[N+j-1]}(a)&=(-1)^{N+j+k}\left(\varphi_{N+1-j}^{[2N-k]}(b)\right)^*,\lb{3.42}
\end{align}
respectively.
For $1\leq j,k\leq N$, we compute using \eqref{2.7c} and \eqref{3.9},
    \begin{align}
        \big[\varphi_k, \varphi_{N+1-j}\big]_Z \Big|^b_a
            &\no=(-1)^N
            \Bigg[
            \sum_{\ell = 0}^{2N-1} (-1)^{1- \ell} \left(\varphi_{N+1-j}^{[2N-1-\ell]}\right)^*\varphi_{k}^{[\ell]}
            \Bigg]
            \Bigg|^b_a\\
            &=(-1)^{N}
            \Bigg[
            \sum_{\ell = 0}^{2N-1} (-1)^{1- \ell} \left(
            \varphi_{N+1-j}^{[2N-1-\ell]}(b)
            \right)^*
            \varphi_{k}^{[\ell]}(b)
            \Bigg]\lb{3.43}\\
            &\quad +(-1)^{N+1}
            \Bigg[
            \sum_{\ell = 0}^{2N-1} (-1)^{1- \ell} \left(
            \varphi_{N+1-j}^{[2N-1-\ell]}(a)
            \right)^*
            \varphi_{k}^{[\ell]}(a)
            \Bigg].\no
    \end{align}
     Since $\varphi_{k}^{[\ell]}(b)=0$ for $0\leq \ell \leq N-1$ and $\varphi_{N+1-j}^{[2N-1-\ell]}(b)=0$ for $N\leq \ell \leq 2N-1$ by \eqref{3.9}, we conclude that the first sum after the second equality in \eqref{3.43} vanishes. So conclusively we observe that
    \begin{align}
        &\big[\varphi_k, \varphi_{N+1-j}\big]_Z \Big|^b_a\lb{3.45}\\
        &\quad= (-1)^{N+1} \Bigg[\sum_{\ell = 0}^{N-1} (-1)^{1- \ell} \left(\varphi_{N+1-j}^{[2N-1-\ell]}(a)\right)^*\varphi_{k}^{[\ell]}(a)\no\\
        &\hspace*{2.5cm} + \sum_{\ell =N}^{2N-1} (-1)^{1- \ell} \left(\varphi_{N+1-j}^{[2N-1-\ell]}(a)\right)^*\varphi_{k}^{[\ell]}(a)\Bigg]\no\\
        &\quad= (-1)^{N+1} \Bigg[\sum_{\ell = 0}^{N-1} (-1)^{1- \ell} \left(\varphi_{N+1-j}^{[2N-1-\ell]}(a)\right)^*\delta_{k,\ell+1}I_M\no\\
        &\hspace*{2.5cm} + \sum_{\ell =N}^{2N-1} (-1)^{1- \ell} \left(\delta_{2N-\ell,N+1-j}I_M\right)^*\varphi_{k}^{[\ell]}(a)\Bigg]\no\\
        &\quad=(-1)^{N+1}\bigg[(-1)^{1-(k-1)} \left(\varphi_{N+1-j}^{[2N-1-(k-1)]}(a)\right)^*+(-1)^{1-(N-1+j)}\varphi_{k}^{[N-1+j]}(a)\bigg]\no\\
        &\quad=(-1)^{N+1}\bigg[(-1)^k \left(\varphi_{N+1-j}^{[2N-k]}(a)\right)^* +(-1)^{N+j}\varphi_{k}^{[N-1+j]}(a)\bigg].\no
        \end{align}
    By Lemma \ref{l2.3} we obtain
    \begin{equation}\lb{3.44Y}
        \big[\varphi_{k},\varphi_{N+1-j}\big]_Z' = \varphi_{N+1-j}^*W \big(\tau_{{}_Z} \varphi_{k}\big) - \big(\tau_{{}_Z} \varphi_{N+1-j}\big)^* W \varphi_{k}=0_M,
    \end{equation}
    since $\tau_{{}_Z}\varphi_{N+1-j}=0_M$ and $\tau_{{}_Z}\varphi_{k}=0_M$. Therefore, in particular, \eqref{3.45} and \eqref{3.44Y} imply
    \begin{align}\lb{3.47}
        0_M &=\big[\varphi_k, \varphi_{N+1-j}\big]_Z \Big|^b_a\no \\
        &=(-1)^{N+1}\bigg[(-1)^k \left(\varphi_{N+1-j}^{[2N-k]}(a)\right)^* +(-1)^{N+j}\varphi_{k}^{[N-1+j]}(a)\bigg].
    \end{align}
    Hence, \eqref{3.47} yields \eqref{3.41} after a simple algebraic manipulation.
    \par
    For \eqref{3.42}, we perform a similar calculation, using \eqref{3.9} once again:
    \begin{align}
        \big[\varphi_{N+k}, \varphi_{N+1-j}\big]_Z\Big |_a^b&=(-1)^N
            \Bigg[
            \sum_{\ell = 0}^{2N-1} (-1)^{1- \ell} \left(\varphi_{N+1-j}^{[2N-1-\ell]}\right)^*\varphi_{N+k}^{[\ell]}
            \Bigg]
            \Bigg|^b_a
            &\\
            &=(-1)^{N}
            \Bigg[
            \sum_{\ell = 0}^{N-1} (-1)^{1- \ell} 0_M
            \varphi_{N+k}^{[\ell]}(b)\no\\
            &\hspace*{2cm}+\sum_{\ell = N}^{2N-1} (-1)^{1- \ell} \left(
            \varphi_{N+1-j}^{[2N-1-\ell]}(b)
            \right)^*
            \delta_{k,\ell+1}I_M
            \Bigg]\no\\
            &\quad+(-1)^{N+1}
            \Bigg[
            \sum_{\ell = 0}^{N-1} (-1)^{1- \ell} \left(
            \varphi_{N+1-j}^{[2N-1-\ell]}(a)
            \right)^*
            0_M\no\\
            &\hspace*{2.7cm}+\sum_{\ell = N}^{2N-1} (-1)^{1- \ell} \left(
            \delta_{N+1-j,2N-\ell}I_M
            \right)^*
            \varphi_{N+k}^{[\ell]}(a)
            \Bigg]\no\\
            &=(-1)^N\bigg[(-1)^{k}\left(\varphi_{N+1-j}^{[2N-k]}(b)\right)^*-(-1)^{N+j}\varphi_{N+k}^{[N-1+j]}(a)\bigg].\no
    \end{align}
    Thus, after another application of Lemma \ref{l2.3}, we have
    \begin{align}\lb{3.50}
        0_M&=\big[\varphi_{N+k}, \varphi_{N+1-j}\big]_Z\Big |_a^b\no\\
        &= (-1)^N\bigg[(-1)^{k}\left(\varphi_{N+1-j}^{[2N-k]}(b)\right)^*-(-1)^{N+j}\varphi_{N+k}^{[N-1+j]}(a)\bigg].
    \end{align}
    Finally, we simplify \eqref{3.50} to obtain \eqref{3.42}, as desired.
\end{proof}
The two matrices $A_{\rm K}$ and $B_{\rm K}$ are invertible.  As a result, $Y(b)$ can be isolated in the boundary condition in \eqref{3.31}.
\begin{proposition}\lb{p3.5}
Assume Hypothesis \ref{h2.1}.  The matrices $A_{\rm K}$ and $B_{\rm K}$ defined by \eqref{3.29} and \eqref{3.30} are invertible. In particular,
\begin{equation}\lb{3.51}
    \dom(H_{Z,{\rm K}}) = \big\{y\in \dom(H_{Z,\max})\,\big|\, Y(b)=B_{\rm K}^{-1}A_{\rm K}Y(a)\big\}.
\end{equation}
\end{proposition}
\begin{proof}
We will show that $B_{\rm K}$ is invertible; the invertibility of $A_{\rm K}$ can be shown with an analogous argument. We define the following vectors:
\begin{equation}\lb{3.52}
    D_{j,k}=\begin{pmatrix}
        \varphi^{[N]}_{N+j,k}(a)\\[1ex]
        \varphi^{[N+1]}_{N+j,k}(a)\\
        \vdots\\[.5ex]
        \varphi^{[2N-1]}_{N+j,k}(a)
    \end{pmatrix}\in (\bbC^M)^{N}, \quad 1\leq j \leq N,\, 1\leq k \leq M,
\end{equation}
which are the columns of the matrix $\Phi_{N}(a)$. We show that $\{D_{j,k}\}_{j=1,k=1}^{N,M}$ is linearly independent. Suppose that there exists $\{d_{j,k}\}_{j=1, k=1}^{N,M}\subset\bbC$ such that
\begin{equation}\lb{3.50z}
    \sum_{j=1,k=1}^{N,M} d_{j,k} D_{j,k} = 0.
\end{equation}
By \eqref{3.52}, the identity in \eqref{3.50z} is equivalent to
\begin{equation}\lb{3.54}
    \sum_{j=1, k=1}^{N,M}d_{j,k}\varphi^{[N+\ell-1]}_{N+j,k}(a)=0,\quad 1\leq \ell \leq N.
\end{equation}
We define the function $\Psi$ over the interval $[a,b]$ by
\begin{equation}
    \Psi=\sum_{j=1, k=1}^{N,M}d_{j,k}\varphi_{N+j,k}.   
\end{equation}
Now, we consider the first $2N-1$ quasi-derivatives of $\Psi$ at $a$. Recall that for $1\leq j \leq N$, $1\leq k \leq M$, $\varphi_{j,k}$ is the unique function in $\ker(H_{Z,\max})$ such that $(\Gamma{\varphi_{j+N,k}})_{\ell}=\delta_{j+N,\ell}e_k$ for all $1 \leq \ell \leq 2N$. From this, we note that for $1\leq j \leq N$, $1\leq k\leq M$, the first $N-1$ quasi-derivatives of $\varphi_{j+N,k}$ evaluated at $a$ are all zero. Since $\Psi$ is simply a linear combination of the $\varphi_{N+j,k}$, it is clear that $\Psi^{[\ell]}(a)=0$. Secondly, we observe that
\begin{equation}
    \Psi^{[N+\ell-1]}(a)=\sum_{j=1, k=1}^{N,M}d_{j,k}\varphi^{[N+\ell-1]}_{N+j,k}(a)=0,\quad 1\leq \ell \leq N,
\end{equation}
by \eqref{3.54}. In total, this yields that the first $2N-1$ quasi-derivatives of $\Psi$ evaluated at $a$ are zero. For our final observation, we note
\begin{equation}
    \tau_{{}_Z}\Psi=\sum_{j=1, k=1}^{N,M}d_{j,k}(\tau_{{}_Z}\varphi_{j,k})=\sum_{j=1,k=1}^{N,M}d_{j,k}(0)=0.   
\end{equation}
Therefore, $\Psi$ is a solution to the initial value problem given by
\begin{equation}\lb{3.58}
    \begin{cases}
    \tau_{{}_Z}f=0,\\
        f^{[\ell-1]}(a)=0,\quad 1\leq \ell\leq 2N.
    \end{cases}
\end{equation}
By \cite[Proposition 2.4]{MZ95a}, solutions to such initial value problems are unique. Since the zero function also satisfies the initial value problem \eqref{3.58}, we have $\Psi=0$. However, $\{\varphi_{j,k}\}_{j=1,k=1}^{2N,M}$ is linearly independent, so $d_{j,k}=0$ for all $j,k$. Therefore, $\left\{D_{j,k}\right\}_{j=1,k=1}^{N,M}$ is linearly independent, and it follows that the matrix $\Phi_N(a)$ is invertible. Taking
    \begin{equation}
    \widehat{B}_{\rm K}:=\left(\begin{array}{c | c}
        {\Phi_N(a)}^{-1} &0_{M,N}\\[1mm]
        \hline\\[-3.5mm]
        \Phi_N(b){\Phi_N(a)}^{-1} &I_{M,N}
    \end{array}\right),
    \end{equation}
one then verifies that
\begin{equation}
\widehat{B}_{\rm K}B_{\rm K} =
\left(\begin{array}{c | c}
        I_{M,N} &0_{M,N}\\
        \hline
        0_{M,N} &I_{M,N}
    \end{array}\right).
\end{equation}
Hence, $B_{\rm K}$ is invertible with $B_{\rm K}^{-1}=\widehat{B}_{\rm K}$.  Finally, \eqref{3.51} follows directly from \eqref{3.31} after a simple algebraic manipulation.
\end{proof}

For completeness, we recall the notion of relatively prime self-adjoint extensions (see \cite[p.~110]{AG81}).

\begin{definition}
    If $S$ is a densely defined symmetric operator and $T$, $T'$ are two self-adjoint extensions of $S$, then the {\it maximal common part} of $T$ and $T'$ is the operator $C_{T,T'}$ defined by
    \begin{equation}
        C_{T,T'}y = Ty, \quad y \in \dom(C_{T,T'}) = \{u \in \dom(T) \cap \dom(T')\ |\ Tu = T'u\}.
    \end{equation}
    In addition, $T$ and $T'$ are {\it relatively prime with respect to $S$} if $C_{T,T'} = S$; that is, if
    \begin{equation}
        \dom(T) \cap \dom(T') = \dom(S).
    \end{equation}
\end{definition}
Given a densely defined strictly positive closed operator $S$, the following abstract result shows that the Krein--von Neumann extension and any strictly positive self-adjoint extension of $S$ are always relatively prime with respect to $S$.
\begin{proposition}\lb{P3.7}
    Suppose $S$ is a densely defined, closed, strictly positive operator in a separable Hilbert space $(\sH,\langle\dott,\dott\rangle_{\sH})$ and let $S_{\rm K}$ denote its Krein von--Neumann extension. If $\widetilde {S}$ is a strictly positive self-adjoint extension of $S$, then $S_{\rm K}$ and $\widetilde{S}$ are relatively prime with respect to $S$.  In particular, $S_{\rm K}$ and the Friedrichs extension $S_{\rm F}$ of $S$ are relatively prime with respect to $S$.
\end{proposition}
\begin{proof}
    We will show $\dom(S_{\rm K})\cap \dom\big(\widetilde{S}\big)=\dom(S)$.
    By \eqref{1.8} in Theorem \ref{T1.1}, for any $u\in\dom(S_{\rm K})$, there exists $\psi\in\dom(S)$ and $\phi\in\ker(S^*)$ such that
    \begin{equation}
        u=\psi+\phi.
    \end{equation}
    If we assume $u\in \dom\big(\widetilde{S}\big)$ also, then we have
    \begin{equation}
        \phi=(u-\psi)\in\dom(\widetilde{S}),
    \end{equation}
    given that $\dom(S)$ is a subspace of $\dom(\widetilde{S})$. Applying the fact that $\phi\in \ker(S^*)\cap\dom(\widetilde{S})$, yields
    \begin{equation}
        \widetilde{S}\phi=S^*\phi=0.
    \end{equation}
    Since we assume that $\widetilde{S}$ is strictly positive, we may conclude that $\phi=0$. Therefore, $u=\psi\in \dom(S)$.
\end{proof}
\begin{remark}\lb{R3.7}
    Although Proposition \ref{P3.7} establishes that $\Hz{\rm F}$ and $\Hz{\rm K}$ are relatively prime with respect to $H_{Z,\min}$, our characterization in Theorem \ref{T3.4} yields a computational proof of the same fact. If $f\in \dom(H_{Z,\rm F})\cap \dom(H_{Z, \rm K})$, combining \eqref{3.4} and \eqref{3.31}, we obtain the boundary condition
    \begin{equation}
        \left( \begin{array}{c|c}
        -\Phi_0(a) & I_{M,N} \\ \hline
        \Phi_0(b) & 0_{M,N}
    \end{array}\right)\begin{pmatrix}
            0 \\
            \vdots \\
            0 \\
            f^{[N]}(a) \\
            \vdots
            \\
            f^{[2N-1]}(a)
        \end{pmatrix} = \left( \begin{array}{c|c}
         \Phi_N(a) & 0_{M,N} \\ \hline
        -\Phi_N(b) & I_{M,N}
    \end{array}\right)\begin{pmatrix}
            0 \\
            \vdots \\
            0 \\
            f^{[N]}(b) \\
            \vdots
            \\
            f^{[2N-1]}(b)
        \end{pmatrix}.
    \end{equation}
    Expanding out, we obtain
    \begin{equation}
        \begin{pmatrix}
            f^{[N]}(a) \\
            \vdots
            \\
            f^{[2N-1]}(a)
        \end{pmatrix}=0=
        \begin{pmatrix}
            f^{[N]}(b) \\
            \vdots
            \\
            f^{[2N-1]}(b)
        \end{pmatrix},
    \end{equation}
    which implies $\dom(H_{Z,\rm K})\cap \dom(H_{Z,\rm F})=\dom(H_{Z,\min})$ by \eqref{2.9}.  Thus, $H_{Z,{\rm K}}$ and $H_{Z,{\rm F}}$ are relatively prime with respect to $H_{Z,\min}$.\hfill$\diamond$
\end{remark}
 
\section{Applications} \lb{s4}
 
Here we consider applications of Theorem \ref{T3.4} and Proposition \ref{p3.5} to generalized four-coefficient Sturm--Liouville expressions, a fourth-order differential expression, and the even order pure differential expression in \eqref{1.13b}.  In particular, we use our abstract approach to recover the characterization by Granovskyi and Oridoroga \cite{GO18a} of the Krein--von Neumann extension corresponding to \eqref{1.13b}.

 
\subsection{Four-coefficient Generalized Sturm--Liouville Operator}
 
In this example, we consider a regular four-coefficient generalized Sturm--Liouville operator with matrix-valued coefficients.  Assuming that the associated minimal operator is strictly positive, we apply the results of Section \ref{s3} to characterize its Krein--von Neumann extension.
 
Let $M\in \bbN$ and $[a,b]\subset \bbR$ be fixed.  Suppose that $p,q,r,s\in \Lp([a,b])^{M\times M}$ satisfy the following conditions:\\[1mm]
$(i)$ $p$ and $r$ are positive definite a.e.~on $[a,b]$\\[1mm]
$(ii)$ $p^{-1},q,r,s\in \Lp^1([a,b])^{M\times M}$\\[1mm]
$(iii)$ $q^*=q$ a.e.~on $[a,b]$
 
\medskip
 
The assumptions in $(i)$--$(iii)$ imply that Hypothesis \ref{h2.1} is satisfied with $N=1$, $W=r$, and
\begin{equation}\lb{star}
Z=
\begin{pmatrix}
-s & p^{-1}\\
q & s^*
\end{pmatrix}
\in \left[\Lp^1([a,b])^{M\times M}\right]^{2\times 2}.
\end{equation}
Note that condition (A2) in Hypothesis \ref{h2.1} is vacuous in the case $N=1$.  By \eqref{2.2} and \eqref{2.3}, the quasi-derivatives corresponding to \eqref{star} of a function $y\in \mathcal{D}_Z^{[2]}([a,b])$ are
\begin{align}
    y^{[1]}&=p[y'+sy],\\
    y^{[2]}&=\left[\big(y^{[1]}\big)'-\big(qy+s^*y^{[1]}\big)\right]\no\\
    &=\left[(p[y'+sy])'-s^*p[y'+sy]-qy\right].\no
\end{align}
In particular, \eqref{star} gives rise to the following quasi-differential expression:
\begin{equation}\lb{star2}
\tau_{{}_Z}y = r^{-1}\big[ -(p[y'+sy])'+s^*p[y'+sy]+qy \big],\quad y\in \mathcal{D}_Z^{[2]}([a,b]).
\end{equation}
The expression \eqref{star2} is a generalization of the four-coefficient generalized Sturm--Liouville expression treated in \cite{EGNT13}, where the coefficients are assumed to be scalar-valued (i.e., $M=1$) and $s$ is assumed to be real-valued.
 
The minimal operator $H_{Z,\min}$ corresponding to \eqref{star2} is given by \eqref{2.9}.  Note that $Z_{1,2}=p^{-1}$ is positive definite a.e.~on $[a,b]$, so $H_{Z,\min}$ is bounded from below by Theorem \ref{t2.5}.  Assuming $H_{Z,\min}$ is strictly positive, Theorem \ref{T3.4} and Proposition \ref{p3.5} can be applied to characterize the Krein--von Neumann extension $H_{Z,\rm K}$ of $H_{Z,\min}$.

Here we will compute the matrices $A_{\rm K}$, $B_{\rm K}$, and $B_{\rm K}^{-1}A_{\rm K}$ appearing in Theorem \ref{T3.4} and Proposition \ref{p3.5}.  Since $H_{Z,\min}$ is strictly positive, we have $\dim(\ker(H_{Z,\max}))=2M$.  Let $\{\varphi_{j,k}\}_{j=1,k=1}^{2,M}$ denote the basis for $\ker(H_{Z,\max})$ guaranteed to exist by Lemma \ref{L3.2}.  Define the $M\times M$ matrix-valued functions $\varphi_1$ and $\varphi_2$ by
\begin{equation}
\varphi_j = (\varphi_{j,1}\,|\,\varphi_{j,2}\,|\,\cdots\,|\,\varphi_{j,M}),\quad j\in \{1,2\},
\end{equation}
so that $\varphi_1(a)=\varphi_2(b)=I_{M}$ and $\varphi_1(b)=\varphi_2(a)=0_M$.  By Theorem \ref{T3.4}, we obtain the following matrices:
\begin{equation}
    A_{\rm K}=\begin{pmatrix}
    -\varphi_1^{[1]}(a) & I_M\\
    \varphi_1^{[1]}(b) & 0_M\\
\end{pmatrix}\quad \text{and}\quad B_{\rm K}=\begin{pmatrix}
    \varphi_2^{[1]}(a) & 0_M\\
    -\varphi_2^{[1]}(b) & I_M\\
\end{pmatrix}.
\end{equation}
Finally, we obtain the inverse of \(B_{\rm K}\) in terms of the inverse of $\varphi_2^{[1]}(a)$ (the invertibility of $\varphi_2^{[1]}(a)$ follows from the proof of Proposition \ref{p3.5}, since $\varphi_2(a)$ plays the role of $\Phi_N(a)$ in this example):
\begin{equation}
    B_{\rm K}^{-1}=\begin{pmatrix}
        \big(\varphi_2^{[1]}(a)\big)^{-1} & 0_M\\
        \varphi_2^{[1]}(b)\big(\varphi_2^{[1]}(a)\big)^{-1} & I_M\\
    \end{pmatrix}.
\end{equation}
Therefore,
\begin{equation}\lb{4.6a}
    \dom(H_{Z,\rm K})=\big\{y\in\dom(H_{Z,\max})\,\big|\,Y(b)=T_{\rm K}Y(a)\big\},   
\end{equation}
where by Proposition \ref{p3.5},
\begin{equation}\lb{4.6}
    T_{\rm K}=B_{\rm K}^{-1}A_{\rm K}=\begin{pmatrix}
        -\big(\varphi_2^{[1]}(a)\big)^{-1}\varphi_1^{[1]}(a) & \big(\varphi_2^{[1]}(a)\big)^{-1}\\[1.5mm]
        \varphi_1^{[1]}(b)-\varphi_2^{[1]}(b)\big(\varphi_2^{[1]}(a)\big)^{-1}\varphi_1^{[1]}(a) & \varphi_2^{[1]}(b)\big(\varphi_2^{[1]}(a)\big)^{-1}\\
    \end{pmatrix}.
\end{equation}
The characterization of the Krein--von Neumann extension given in \eqref{4.6a} and \eqref{4.6} generalizes the result in \cite[Theorem 12.3]{EGNT13} to the case of matrix-valued coefficients $p,q,r,s$. In fact, one can see that \eqref{4.6} is equivalent to the form presented in \cite[Theorem 12.3]{EGNT13} when $p,q,r,s$ are scalar-valued (cf.~\eqref{1.12b}):
\begin{equation}
    T_{\rm K}=B^{-1}_{\rm K}A_{\rm K}=\frac{1}{u_2^{[1]}(a)}
    \begin{pmatrix}
        -u_1^{[1]}(a) & 1\\[1.5mm]
        u_2^{[1]}(a)u_1^{[1]}(b)-u_2^{[1]}(b)u_1^{[1]}(a) & u_2^{[1]}(b)\\
    \end{pmatrix},  
\end{equation}
where $u_j$, $j\in \{1,2\}$, are solutions to $\tau_Z y=0$ that satisfy the boundary conditions $u_1(a)=u_2(b)=1$ and $u_1(b)=u_2(a)=0$.

\subsection{Fourth-Order Differential Operator}\lb{s4.2}
 
Let $M=N=1$, $[a,b]\subset\bbR$, and $W=1$ a.e.~on $[a,b]$.  Consider the matrix-valued function $Z=(Z_{j,k})_{j,k=1}^4\in \Lp^1([a,b])^{4\times4}$ given by
\begin{equation}
    Z=
    \begin{pmatrix}
        0 & 1 & 0 & 0 \\
        0 & 0 & 1 & 0 \\
        0 & 0 & 0 & 1 \\
        -1 & 0 & 0 & 0
    \end{pmatrix}\,\text{ a.e.~on $[a,b]$}.
\end{equation}
One then verifies that Hypothesis \ref{h2.1} holds, and the corresponding differential expression is
\begin{equation}\lb{4.11b}
\tau_{{}_Z}y = y^{(4)}+y,\quad y\in \mathfrak{D}_Z^{[4]}([a,b]).
\end{equation}
The minimal operator $H_{Z,\min}$ corresponding to \eqref{4.11b} is given by \eqref{2.8} and is strictly positive.  Since $Z_{2,3}=1>0$ a.e.~on $[a,b]$, Theorem \ref{T3.4} and Proposition \ref{p3.5} can be applied to characterize the Krein--von Neumann extension $H_{Z,\rm K}$ of $H_{Z,\min}$.  Here we will compute the matrices $A_{\rm K}$, $B_{\rm K}$, and $B_{\rm K}^{-1}A_{\rm K}$ appearing in Theorem \ref{T3.4} and Proposition \ref{p3.5}.

To determine $A_{\rm K}$ and $B_{\rm K}$, we must determine the basis $\{\varphi_{k,1}\}_{k=1}^4$ shown to exist in Lemma \ref{L3.2}.  Standard solution methods for linear differential equations imply that the solution space of $\tau_{{}_Z} y=0$ on the interval $[a,b]$ has the basis
\begin{equation}\lb{4.12e}
    \big\{ e^{\omega x},\, e^{-\omega x},\, e^{\omega^3x},\, e^{-\omega^3 x}\big\},
\end{equation}
where $\omega=\frac{1}{\sqrt{2}}+\frac{i}{\sqrt{2}}$ denotes a primitive fourth root of $-1$. We denote the elements in \eqref{4.12e} by $y_1$, $y_2$, $y_3$, and $y_4$, respectively.
 
We must find the basis $\{\varphi_{k,1}\}_{k=1}^4$ of this solution space with $\Gamma \varphi_{k,1} = e_k$, where $e_k=(\delta_{j,k})_{j=1}^4$ denotes the $k$th standard basis vector in $\bbC^4$.  To find it, we need to find coefficients $\{c_{j,k}\}_{j,k=1}^4\subset \bbC$ such that
\begin{equation}
    \varphi_{\ell,1}=\sum_{j=1}^4 c_{j,\ell}y_j,\quad 1\leq \ell\leq 4.
\end{equation}
The set of equations $\Gamma \varphi_{\ell,1}=e_{\ell}$, $1\leq \ell\leq 4$, can then be arranged into a matrix equation:
\begin{equation}
    \left(\, \Gamma y_1 \,|\, \Gamma y_2 \,|\, \Gamma y_3 \,|\, \Gamma y_4
    \,\right)(c_{j,k})_{j,k=1}^4 = I_4.
\end{equation}
Thus,
\begin{equation}
    (c_{j,k})_{j,k=1}^4=\left(\, \Gamma y_1 \,|\, \Gamma y_2 \,|\, \Gamma y_3 \,|\, \Gamma y_4
    \,\right)^{-1}.
\end{equation}
So to determine the coefficients of $\varphi_{\ell,1}$ for each $1\leq \ell\leq 4$, it suffices to invert the matrix
\begin{equation}
\Lambda:=\left(\, \Gamma y_1 \,|\, \Gamma y_2 \,|\, \Gamma y_3 \,|\, \Gamma y_4
    \,\right)
    \end{equation}
and read off its $\ell$th column.

On the interval $[a,b]$, we obtain
\begin{equation}\lb{4.15}
    \Lambda=
    \begin{pmatrix}
    e^{\omega a} &e^{-\omega a} &e^{\omega^3 a} &e^{-\omega^3 a} \\
    \omega e^{\omega a} & -\omega e^{-\omega a}& \omega^3 e^{\omega^3 a}& -\omega^3 e^{-\omega^3 a}\\
    e^{\omega b} &e^{-\omega b} &e^{\omega^3 b} &e^{-\omega^3 b} \\
    \omega e^{\omega b} & -\omega e^{-\omega b}& \omega^3 e^{\omega^3 b}& -\omega^3 e^{-\omega^3 b}\\
\end{pmatrix}.\end{equation}
To simplify expressions, we will assume $a=0$ and $b=\sqrt{2}\pi$ throughout the remainder of this example. On this interval, \eqref{4.15} is of the form
\begin{equation}\lb{4.16}
\Lambda=\begin{pmatrix}
1 & 1 & 1 & 1 \\
\omega & -\omega & \alpha^3 & -\alpha^3 \\
-e^{\pi} & -e^{-\pi} & -e^{-\pi} & -e^{\pi} \\
-\omega e^{\pi} & \omega e^{-\pi} & -\alpha^3e^{-\pi} & \alpha^3e^{\pi}
\end{pmatrix},\end{equation}
where $\alpha:=e^{\pi}$.  The inverse of \eqref{4.16} is easily computed by hand or by computer algebra, yielding the matrix
\begin{equation}\lb{4.17}
\Lambda^{-1}=\frac{\omega}{\sqrt{2}(\alpha^2-1)}\begin{pmatrix}
-1 & \omega^3 & -\alpha & \omega^3 \alpha \\
\alpha^2 & \omega^3 \alpha^2 & \alpha & \omega^3 \alpha \\
-i\alpha^2 & -\omega^3 \alpha^2 & -i\alpha & -\omega^3 \alpha \\
i & -\omega^3 & i\alpha & -\omega^3\alpha
\end{pmatrix}.
\end{equation}
Reading off the columns of \eqref{4.17}, we obtain the four functions
\begin{equation}
\begin{cases}
\varphi_{1,1}(x) = \displaystyle{\frac{\omega}{\sqrt{2}(\alpha^2-1)}}\big[-e^{\omega x}+\alpha^2e^{-\omega x} +i(-\alpha^2e^{\omega^3 x} + e^{-\omega^3 x}) \big],\\[4mm]
\varphi_{2,1}(x) = \displaystyle{-\frac{1}{\sqrt{2}(\alpha^2-1)}}\big[ e^{\omega x}+ \alpha^2e^{-\omega x}- \alpha^2e^{\omega^3x} -e^{-\omega^3 x}\big],\\[4mm]
\varphi_{3,1}(x) = \displaystyle{\frac{\alpha\omega}{\sqrt{2}(\alpha^2-1)}}\big[-e^{\omega x}+e^{-\omega x} +i(-e^{\omega^3 x}+e^{-\omega^3 x})\big],\\[4mm]
\varphi_{4,1}(x) = \displaystyle{-\frac{\alpha}{\sqrt{2}(\alpha^2-1)}}\big[e^{\omega x}+e^{-\omega x}-e^{\omega^3 x}- e^{-\omega^3 x} \big],\quad x\in[0,\sqrt{2}\pi].
\end{cases}
\end{equation}
By taking second and third derivatives and evaluating at $0$ and $\sqrt{2}\pi$, we compute the boundary condition matrices as
\begin{equation}
    A_{\rm K}=
    \begin{pmatrix}
    1 & \sqrt{2}i\frac{\alpha^2+1}{\alpha^2-1} & 1 & 0 \\
    -\sqrt{2}\frac{\alpha^2+1}{\alpha^2-1} & -i & 0 & 1 \\
    0 & \frac{2\sqrt{2}\alpha i}{\alpha^2-1} & 0 & 0 \\
    -2\sqrt{2} \frac{\alpha}{\alpha^2-1} & 0 & 0 & 0
    \end{pmatrix},
\end{equation}
\begin{equation}
    B_{\rm K}=
    \begin{pmatrix}
    0 & -2\sqrt{2}i\frac{\alpha}{\alpha^2-1} & 0 & 0 \\
    \frac{2\sqrt{2}\alpha}{\alpha^2-1} & 0 & 0 & 0 \\
    1 & -\sqrt{2}i\frac{\alpha^2+1}{\alpha^2-1} & 1 & 0 \\
    \sqrt{2}\frac{\alpha^2+1}{\alpha^2-1} & -i & 0 & 1
    \end{pmatrix}.
\end{equation}
Using hyperbolic trigonometric identities, we have
\begin{equation}
B^{-1}_{\rm K}=\begin{pmatrix}
0 & \frac{\sinh(\pi)}{\sqrt{2}} & 0 & 0 \\
\frac{i\sinh(\pi)}{\sqrt{2}} & 0 & 0 & 0\\
-\cosh(\pi) & -\frac{\sinh(\pi)}{\sqrt{2}} & 1 & 0 \\
-\frac{\sinh(\pi)}{\sqrt{2}} & -\cosh(\pi) & 0 & 1
\end{pmatrix}.
\end{equation}
Finally, we compute $B^{-1}_{\rm K}A_{\rm K}$ and obtain
\begin{equation}
B^{-1}_{\rm K}A_{\rm K}=\begin{pmatrix}
-\cosh(\pi) & -\frac{i\sinh(\pi)}{\sqrt{2}} & 0 & \frac{\sinh(\pi)}{\sqrt{2}} \\
\frac{i\sinh(\pi)}{\sqrt{2}} & -\cosh(\pi) & \frac{i\sinh(\pi)}{\sqrt{2}} & 0 \\
0 & -\frac{i\sinh(\pi)}{\sqrt{2}} & -\cosh(\pi) & -\frac{\sinh(\pi)}{\sqrt{2}} \\
\frac{\sinh(\pi)}{\sqrt{2}} & 0 & -\frac{\sinh(\pi)}{\sqrt{2}} & -\cosh(\pi)
\end{pmatrix}.
\end{equation}
 
\subsection{Pure Differential Operator of Order $2N$}
Let $M=1$, $N\in \bbN$, $[a,b]\subset\bbR$, and $W=1$ a.e.~on $[a,b]$.  Consider $Z=(Z_{j,k})_{j,k=1}^{2N}\in \Lp^1([a,b])^{2N\times 2N}$ defined by
\begin{equation}\lb{4.23b}
Z_{j,k}=\delta_{j+1,k}\, \text{ a.e.~on $[a,b]$ for $1\leq j,k\leq 2N$}.
\end{equation}
That is, $Z$ is almost everywhere constant on $[a,b]$:
\begin{equation}\lb{4.24b}
    Z=\begin{pmatrix}
        0&1&0&\dots&0&0\\
        0&0&1&\dots&0&0\\
        \vdots&\vdots&&\ddots&&\vdots\\
        0&0&0&\dots&1&0\\
        0&0&0&\dots&0&1\\
        0&0&0&\dots&0&0
    \end{pmatrix}\,\text{ a.e.~on $[a,b]$.}
\end{equation}
One then verifies that Hypothesis \ref{h2.1} holds.  By \eqref{2.2} and \eqref{2.3}, the quasi-derivatives corresponding to \eqref{4.24b} of a function $y\in \mathfrak{D}_Z^{[2N]}([a,b])$ are simply ordinary derivatives:
\begin{equation}
y^{[j]}=y^{(j)},\quad 0\leq j\leq 2N.
\end{equation}
Thus, \eqref{4.24b} gives rise to the following (formally nonnegative) pure ordinary differential expression of order $2N$:
\begin{equation}
\tau_{{}_Z}y=(-1)^Ny^{(2N)},\quad y\in \mathfrak{D}_Z^{[2N]}([a,b]),
\end{equation}
and $\tau_{{}_Z}$ generates the minimal operator $H_{Z,\min}$ in accordance with \eqref{2.9}.  The coefficient $Z_{N,N+1}=1>0$ a.e.~on $[a,b]$, and $H_{Z,\min}$ is strictly positive.  We will apply Theorem \ref{T3.4} and Proposition \ref{p3.5} to characterize the domain of the Krein--von Neumann extension $H_{Z,\rm K}$ of $H_{Z,\min}$.

The calculations throughout this subsection make extensive use of several combinatorial identities.  For completeness, we recall these identities in the following lemma and refer to \cite[Pages 58--59, Equations (17), (22), \& (23)]{Kn97} for further details. We employ the combinatorial convention (see, e.g., \cite[Page 55, Equation (6)]{Kn97}) that for nonnegative real numbers $r$ and all positive integers $k$,
\begin{equation}
    \binom{r}{-k}=\binom{r}{r+k}=0,
\end{equation}
which implies, upon taking $r=0$, that
\begin{equation}
    \frac{1}{(-k)!}=0.
\end{equation}
\begin{lemma}\lb{L4.1}
The following combinatorial identities hold:\\
$(i)$ If $p,q \in \bbZ$, then
            \begin{equation}\lb{4.20}
            \binom{-p}{q}=(-1)^{q}\binom{p+q-1}{q}.   
            \end{equation}
$(ii)$ If $s \in \bbR, r \in \bbZ_{\geq 0}, m,n \in \bbZ$, then,
            \begin{equation}\lb{4.21}
            \sum_{k=-\infty}^{\infty} \binom{r}{m+k} \binom{s}{n+k}= \binom{r+s}{r-m+n}.
            \end{equation}
$(iii)$ If $N \in \bbZ_{\geq 0}$, $1\leq j,k \leq N$, then
            \begin{equation}\lb{4.22}
            \sum_{\ell=1}^{N}(-1)^{\ell+k}{\binom{j-1}{\ell-1}} \binom{\ell-1}{k-1}=\delta_{j,k}.
            \end{equation}
\end{lemma}
In order to apply Theorem \ref{T3.4} to characterize $H_{Z,\rm K}$, we must determine the basis $\{\varphi_{k,1}\}_{k=1}^{2N}$ of $\ker(H_{Z,\max})$ that was shown to exist in Lemma \ref{L3.2}.  By integration, it is clear that
\begin{equation}\lb{4.27b}
\ker(H_{Z,\max}) = {\rm span}\, \{x^{j-1}\}_{j=1}^{2N}.
\end{equation}
We shall first calculate $\{\varphi_{k,1}\}_{k=1}^{2N}$ in the special case when $[a,b]=[0,1]$: let us relabel the basis $\{\varphi_{k,1}\}_{k=1}^{2N}$ as $\{p_k\}_{k=1}^{2N}$ in this case.  Thus, in the case $[a,b]=[0,1]$, we must determine $\{p_k\}_{k=1}^{2N}\subset {\rm span}\, \{x^{j-1}\}_{j=1}^{2N}$ such that for each $1\leq j\leq N$,
\begin{equation}
    \begin{split}
        p_k^{(j-1)}(0)&=\delta_{j,k}, \quad p_k^{(j-1)}(1)=0, \hspace*{1.2cm} 1\leq k\leq N,\\[1mm]
        p_k^{(j-1)}(0)&=0, \hspace*{.7cm} p_k^{(j-1)}(1)=\delta_{j+N,k}, \quad N+1\leq k\leq2N.
    \end{split}
\end{equation}

Lemma \ref{L3.2} guarantees, for each $1\leq k\leq 2N$, the existence of $p_k\in \ker(H_{Z,\max})$ with
\begin{equation}
\Gamma p_k =
\begin{pmatrix}
              p(0)\\
              p^{(1)}(0)\\
              \vdots\\
              p^{(N-1)}(0)\\
              p(1)\\
              p^{(1)}(1)\\
              \vdots\\
              p^{(N-1)}(1)
\end{pmatrix}=e_k,
\end{equation}
where $e_k=(\delta_{j,k})_{j=1}^{2N}$ denotes the $k$th standard basis vector in $\bbC^{2N}$.  In light of \eqref{4.27b}, there exist scalars $\{c_{\ell,k}\}_{\ell=1}^{2N}\subset\bbC$ such that
\begin{equation}\lb{4.32}
    p_k=\sum_{\ell=1}^{2N}c_{\ell, k}x^{\ell-1},\quad 1\leq k\leq 2N.
\end{equation}
Therefore,
\begin{equation}
    e_k=\Gamma{p_k}=\Gamma\left(\sum_{\ell=1}^{2N}c_{\ell, k}x^{\ell-1}\right)=\sum_{\ell=1}^{2N}c_{\ell, k}\Gamma x^{\ell-1},\quad 1\leq k\leq 2N,
\end{equation}
which can be recast as a matrix product
\begin{equation}\lb{4.33}
    \Lambda \begin{pmatrix}
    c_{1,k}\\
    c_{2,k}\\
    \vdots\\
    c_{2N,k}
    \end{pmatrix}=e_k,\quad 1\leq k\leq 2N,
\end{equation}
where
\begin{equation}
\Lambda := \left(\Gamma1\,|\,\Gamma x\,|\,\cdots\,|\,\Gamma x^{2N-1}\right)\in \bbC^{2N\times 2N}.
\end{equation}
The equalities in \eqref{4.33} may be summarized in matrix form as
\begin{align}
	\Lambda\begin{pmatrix}
        c_{1,1}&c_{1,2}&\dots&c_{1,2N}\\
        c_{2,1}&c_{2,2}&\dots&c_{2,2N}\\
        \vdots&\vdots&\ddots&\vdots\\
        c_{2N,1}&c_{2N,2}&\dots&c_{2N,2N}
    \end{pmatrix}=\left(e_1\,|\,e_2\,|\,\cdots\,|\,e_{2N}\right)=I_{2N}.
\end{align}
Thus, we deduce that
\begin{equation}\lb{4.36}
    \begin{pmatrix}
        c_{1,1}&c_{1,2}&\dots&c_{1,2N}\\
        c_{2,1}&c_{2,2}&\dots&c_{2,2N}\\
        \vdots&\vdots&\ddots&\vdots\\
        c_{2N,1}&c_{2N,2}&\dots&c_{2N,2N}
\end{pmatrix}=\Lambda^{-1}.
\end{equation}
Therefore, to retrieve the basis $\{p_k\}_{k=1}^{2N}$ for $\ker(H_{Z,\max})$ it is crucial to understand the matrix $\Lambda^{-1}$.  We observe that $\Lambda$ has a $2 \times 2$ block matrix form:
\begin{equation}
    \Lambda = \left(\begin{array}{c|c}
    A & 0_N \\
    \hline
    C & D
\end{array}\right)
\end{equation}
in which
\begin{align}
    A  &= \left(\frac{d^{j-1}}{dx^{j-1}}\big[x^{k-1}\big]\Big|_{x=0}\right)_{j,k=1}^N=\big(\delta_{j,k}(j-1)!\big)_{j,k=1}^N,\\
    C &= \left(\frac{d^{j-1}}{dx^{j-1}}\big[x^{k-1}\big]\Big|_{x=1}\right)_{j,k=1}^N=\left(\frac{(k-1)!}{(k-j)!}\right)_{j,k=1}^N,\\
    D &= \left(\frac{d^{j-1}}{dx^{j-1}}\big[x^{N+k-1}\big]\Big|_{x=1}\right)_{j,k=1}^N=\left(\frac{(N+k-1)!}{(N+k-j)!}\right)_{j,k=1}^N.\lb{4.41c}
\end{align}
It is clear that $A$ is invertible and that
\begin{equation}\lb{4.42}
    A^{-1} =\left(\delta_{j,k}\frac{1}{(j-1)!}\right)_{j,k=1}^N.
\end{equation}
Next, we show that $D$ is invertible and obtain an explicit form for $D^{-1}$.
\begin{lemma}\lb{L4.2}
    The matrix $D$ defined by \eqref{4.41c} is invertible and
    \begin{equation}\lb{4.43}
        D^{-1} =
        \left(\sum_{\ell=1}^N \frac{(-1)^{j+\ell}}{(k-1)!}
        \binom{\ell-1}{j-1}\binom{N-1+\ell-k}{\ell-k}\right)_{j,k=1}^N.
    \end{equation}
\end{lemma}
\begin{proof}
    It suffices to prove that
    \begin{equation}\lb{4.44}
        QA^{-1}D = P
    \end{equation}
    in which
    \begin{equation}\lb{4.48c}
        Q = \left((-1)^{j-k}\binom{N-1+j-k}{j-k}\right)_{j,k=1}^N \quad \text{ and } \ P = \left(\binom{k-1}{j-1}\right)_{j,k=1}^N,
    \end{equation}
    where $Q$ and $P$ arise naturally when one performs row reduction operations on $D$. In fact, the matrix $P$, whose entries give Pascal's triangle in the upper triangle of the matrix, is invertible with the inverse
    \begin{equation}\lb{4.49c}
        P^{-1}= \Bigg((-1)^{j+k}\binom{k-1}{j-1}\Bigg)_{j,k=1}^N,
    \end{equation}
    as one can verify via Lemma \ref{L4.1} part $(iii)$.  Thus, \eqref{4.44} implies
    \begin{equation}\lb{4.50c}
        (P^{-1}QA^{-1})D=I_N,
    \end{equation}
    which proves the lemma. To show \eqref{4.44}, we observe that
    \begin{align}
        (QA^{-1}D)_{j,k}
        &=\sum_{s,\ell = 1}^N (-1)^{j+s} \binom{N-1+j-s}{j-s} \frac{\delta_{s,\ell}(k+N-1)!}{(s-1)!(k+N-\ell)!}\no\\
        &=\sum_{s= 1}^N (-1)^{j+s} \binom{N-1+j-s}{j-s} \frac{(k+N-1)!}{(s-1)!(k+N-s)!}\no\\
        &=\sum_{s=1}^N (-1)^{j+s} \binom{N-1+j-s}{j-s} \binom{N+k-1}{s-1}.\lb{4.55d}
    \end{align}
    We observe that the second combination in \eqref{4.55d} is only nonzero for $1\leq s\leq N+k$ and the first is only nonzero for $s\leq j\leq N$. Therefore, all the nonzero terms for integer $s$ fall between $1$ and $N$. Thus, we reindex the sum in the final expression in \eqref{4.55d} to be over all the integers, and the calculation may be continued as follows:
    \begin{align}
    &\sum_{s=-\infty}^{\infty} (-1)^{j+s} \binom{N-1+j-s}{j-s} \binom{N+k-1}{s-1}\lb{4.55e}\\
        &\quad=\sum_{s=-\infty}^{\infty} \binom{-N}{-N-j+s} \binom{N+k-1}{s-1}\no\\
        &\quad=\binom{N+k-1-N}{N+k-1+1-N-j} \no\\
        &\quad=\binom{k-1}{k-j}\no\\
        &\quad=\binom{k-1}{j-1}\no\\
        &\quad=P_{j,k},\quad 1 \leq j,k \leq N,\no
    \end{align}
    as desired.  In \eqref{4.55e}, the first and second equalities follow from Lemma \ref{L4.1} parts $(i)$ and $(ii)$ , respectively.  The equality in \eqref{4.43} follows by using \eqref{4.42}, \eqref{4.48c}, and \eqref{4.49c} to compute matrix elements in $D^{-1}=P^{-1}QA^{-1}$.
    \end{proof}
It is now a straightforward calculation to verify that $\Lambda^{-1}$ is given, in block form, by
\begin{equation}\lb{4.50b}
    \Lambda^{-1}= \left(\begin{array}{c|c}
        A & 0_N \\
        \hline
        C & D
    \end{array}\right)^{-1} = \left(\begin{array}{c|c}
        A^{-1} & 0_N \\
        \hline
        -D^{-1}CA^{-1} & D^{-1}
    \end{array}\right),
\end{equation}
where $A^{-1}$ and $D^{-1}$ are given in \eqref{4.42} and \eqref{4.43}, respectively, and
\begin{equation} \lb{4.49}
	\begin{split}
		(D^{-1}CA^{-1})_{j,k}=
		\sum_{r,\ell=1}^N
		\frac{(-1)^{j+r}}{(r-1)!(k-r)!}\binom{\ell-1}{j-1}\binom{N+\ell-r-1}{N-1},&\\
		1\leq j,k\leq N.
	\end{split}
\end{equation}

Recalling \eqref{4.32} and \eqref{4.36}, for each $1\leq k\leq 2N$, we obtain an explicit form for each polynomial $p_k$ by reading off the $k$th column $(c_{\ell,k})_{\ell=1}^{2N}$ of $\Lambda^{-1}$ and writing the linear combination in \eqref{4.32}.  For $1\leq k\leq N$, we obtain:
    \begin{align}
        p_k &=\frac{x^{k-1}}{(k-1)!}-
         \sum_{r,\ell,s=1}^N
        \frac{(-1)^{s+r}x^{N-1+s}}{(r-1)!(k-r)!} \binom{\ell-1}{s-1}\binom{N+\ell-r-1}{N-1},\no\\
    p_{N+k} &= \sum_{s,\ell=1}^N \frac{(-1)^{s+k}x^{N-1+s}}{(k-1)!} \binom{\ell-1}{s-1}\binom{N+\ell-k-1}{N-1},
    \end{align}
which yields $\{p_k\}_{k=1}^{2N}$.

Returning to the general case of arbitrary $[a,b]$, the basis $\{\varphi_{k,1}\}_{k=1}^{2N}$ for the subspace $\ker(H_{Z,\max})$ may be obtained from $\{p_k\}_{k=1}^{2N}$ by scaling and translation via
\begin{equation}\lb{4.30}
    \varphi_{k,1}(x)=
        \begin{cases}
            (b-a)^{k-1}p_k\left(\displaystyle{\frac{x-a}{b-a}}\right),& 1\leq k\leq N,\\[3mm]
            (b-a)^{k-N-1}p_k\left(\displaystyle{\frac{x-a}{b-a}}\right),& N+1\leq k\leq 2N,
    \end{cases}\quad x\in[a,b].
\end{equation}
In this way, the only nonzero derivative of $\varphi_{k,1}$ is either the $(k-1)$st derivative at $x=a$ or the $(k-N-1)$st derivative at $x=b$, and by the chain rule, for $1\leq j\leq N$,
\begin{equation}
    \begin{split}
        \vhi_{k,1}^{(j-1)}(a)&=\delta_{j,k},\hspace*{1cm} 1\leq k\leq N,\\[1mm]
        \vhi_{k,1}^{(j-1)}(b)&=\delta_{j+N,k},\quad N+1\leq k\leq 2N.
    \end{split}
\end{equation}
Explicitly, for $1 \leq k \leq N$, we obtain:
    \begin{align}
        \varphi_{k,1} &=\frac{(x-a)^{k-1}}{(k-1)!}\no\\
        &\quad- 
    \sum_{s,r,\ell=1}^N \frac{(-1)^{s+r}(x-a)^{N-1+s}}{(r-1)!(k-r)!(b-a)^{N-k+s}} \binom{\ell-1}{s-1}\binom{N+\ell-r-1}{N-1},\no\\
    \varphi_{N+k,1} &= 
    \sum_{s,\ell=1}^N \frac{(-1)^{s+k}(x-a)^{N-1+s}}{(k-1)!(b-a)^{N-k+s}} \binom{\ell-1}{s-1}\binom{N+\ell-k-1}{N-1}.
    \end{align}
Thus, the matrices $A_{\rm K}$ and $B_{\rm K}$ may be characterized as follows:
    \begin{equation}\lb{4.55b}
    A_{\rm K} =
    \left( \begin{array}{c|c}
        -\Phi_{0}(a) & I_N \\
        \hline
        \Phi_{0}(b) & 0_N \\
    \end{array}\right)\quad\text{and}\quad
    B_{\rm K} =
    \left( \begin{array}{c|c}
        \Phi_{N}(a) & 0_N \\
        \hline
        -\Phi_{N}(b) & I_N \\
    \end{array}\right)
    \end{equation}
in which
    \begin{align}
        \Phi_0(a)
            &\no = \Bigg(
                \frac{-(N-1+j)!}{(b-a)^{N-k+j}}
                \sum_{r,\ell=1}^N
                \bigg(
                    \frac{(-1)^{j+r}}{(r-1)!(k-r)!}
            \\
            &\hspace*{4.8cm}
                    \times\binom{\ell-1}{j-1}\binom{N+\ell-r-1}{N-1}
                \bigg)
            \Bigg)_{j,k=1}^N, \lb{4.53}&\\
        \Phi_0(b)
            &\no = \Bigg(
                \frac{-1}{(b-a)^{N-k+j}}
                \sum_{s,r,\ell=1}^N
                \bigg(
                    \frac{(-1)^{r+s}(N-1+s)!}{(r-1)!(k-r)!(s-j)!}
            \\
            &\hspace*{5.3cm}
                    \times\binom{\ell-1}{j-1}\binom{N+\ell-r-1}{N-1}
                \bigg)
            \Bigg)_{j,k=1}^N, \lb{4.54}&\\
        \Phi_N(a)
            &\no= \Bigg(
                \frac{(N-1+j)!}{(b-a)^{N-k+j}}
                \sum_{\ell=1}^N
                \bigg(
                    \frac{(-1)^{j+k}}{(k-1)!}
            \\
            &\hspace*{4.6cm}
                    \times\binom{\ell-1}{j-1}\binom{N+\ell-k-1}{N-1}
                \bigg)
            \Bigg)_{j,k=1}^N, \lb{4.55}&\\
        \Phi_N(b)
            &\no= \Bigg(
                \frac{1}{(b-a)^{N-k+j}}
                \sum_{s,\ell=1}^N
                \bigg(
                    \frac{(-1)^{s+k}(N-1+s)!}{(k-1)!(s-j)!}
            \\
            &\hspace*{4.7cm}
                    \times\binom{\ell-1}{j-1}\binom{N+\ell-k-1}{N-1}
                \bigg)
            \Bigg)_{j,k=1}^N. \lb{4.56}&
    \end{align}
Define $T_{\rm K}$ to be the upper triangular Toeplitz matrix
\begin{equation}\lb{4.57}
    T_{\rm K} = \left(\frac{(b-a)^{k-j}}{(k-j)!}\right)_{j,k=1}^{2N}.
\end{equation}
We now show that
\begin{equation}\lb{4.60b}
A_{\rm K} = B_{\rm K}T_{\rm K}.
\end{equation}
Observe that $T_{\rm K}$ can be represented as a $2 \times 2$ block matrix
\begin{equation}\lb{4.62b}
    T_{\rm K} = \left(\begin{array}{c|c}
        T_1 & T_2  \\
        \hline
        0 & T_1
    \end{array}\right)
\end{equation}
in which
\begin{equation}\lb{4.59}
    T_1 = \left(\frac{(b-a)^{k-j}}{(k-j)!}\right)_{j,k=1}^{N} \, \text{ and } \
    T_2 = \left(\frac{(b-a)^{k-j}}{(N+k-j)!}\right)_{j,k=1}^{N}.
\end{equation}
Thus, using \eqref{4.55b} and \eqref{4.62b}, $B_{\rm K}T_{\rm K}$ may be computed via blockwise matrix multiplication:
\begin{equation}\lb{4.64b}
B_{\rm K}T_{\rm K}=\left(\begin{array}{c|c}
    \Phi_N(a)T_1 & \Phi_N(a)T_2 \\
     \hline
     -\Phi_N(b)T_1& -\Phi_N(b)T_2 +T_1
\end{array} \right).
\end{equation}
Given \eqref{4.64b} and the $2\times 2$ block structure of $A_{\rm K}$ (cf.~\eqref{4.55b}), the equality in \eqref{4.60b} translates to the following four equalities:
\begin{align}
	-\Phi_N(a)T_1 &= \Phi_0(a),\lb{4.59a}\\
	-\Phi_N(b)T_1 &= \Phi_0(b),\lb{4.60}\\
	\Phi_N(a)T_2 &= I_N,\lb{4.61}\\
	\Phi_N(b)T_2 &= T_1.\lb{4.62}
\end{align}
The identities in \eqref{4.59a} and \eqref{4.60} are immediately evident when the matrix product is written out for a general component. Beginning with \eqref{4.59a}, we have
\begin{align}
    &\big(-\Phi_N(a)T_1\big)_{j,k} \\
    &\no\quad
        =-\sum_{r=1}^{N}(\Phi_N(a))_{j,r}(T_1)_{r,k}\\
    &\no\quad
        =-\sum_{r=1}^{N}\left[
        \frac{(N-1+j)!}{(b-a)^{N-r+j}} \sum_{\ell=1}^N\left(
            \frac{(-1)^{j+r}}{(r-1)!}\binom{\ell-1}{j-1}\binom{N+\ell-r-1}{N-1}
        \right)
        \frac{(b-a)^{k-r}}{(k-r)!}
    \right]\\
    &\no\quad
        =\frac{-(N-1+j)!}{(b-a)^{N-k+j}}\sum_{r,\ell=1}^N\left[
        \frac{(-1)^{j+r}}{(r-1)!(k-r)!}\binom{\ell-1}{j-1}\binom{N+\ell-r-1}{N-1}
        \right]\\
    &\no\quad
        =\big(\Phi_0(a)\big)_{j,k},\quad 1\leq j,k\leq N,
\end{align}
and then for \eqref{4.60}:
\begin{align}
    &\big(-\Phi_N(b)T_1\big)_{j,k} \\
    &\no\quad=
    -\sum_{r=1}^{N}(\Phi_N(a))_{j,r}(T_1)_{r,k}\\
    &\no\quad=-\sum_{r=1}^{N}
    \Bigg[
        \frac{1}{(b-a)^{N-r+j}} \sum_{\ell, s=1}^N
            \bigg(
            \frac{(-1)^{s+r}(N-1+s)!}{(r-1)!(s-j)!}
    \\
    &\no \hspace*{5.7cm}
            \times\binom{\ell-1}{j-1}\binom{N+\ell-r-1}{N-1}
        \bigg)
        \frac{(b-a)^{k-r}}{(k-r)!}
    \Bigg]\\
    &\no\quad
    =\frac{-1}{(b-a)^{N-k+j}}\sum_{r,s,\ell=1}^N\left[
        \frac{(-1)^{s+r}(N-1+s)!}{(r-1)!(k-r)!(s-j)!}\binom{\ell-1}{j-1}\binom{N+\ell-r-1}{N-1}
    \right]\\
    &\no \quad=\left(\Phi_0(b)\right)_{j,k},\quad 1\leq j,k\leq N.
\end{align}
The justifications for \eqref{4.61} and \eqref{4.62} make extensive use of the combinatorial identities in Lemma \ref{L4.1}.  To show that \eqref{4.61} holds, we compute as follows:
\begin{align}
     &\big(\Phi_N(a)T_2\big)_{j,k}\lb{4.77}\\ 
    &\no\quad=\sum_{r=1}^{N}(\Phi_N(a))_{j,r}(T_2)_{r,k}\\ \no
    &\quad=\sum_{r=1}^{N}\bigg[
        \frac{(N-1+j)!}{(b-a)^{N-r+j}}\\ \no &\hspace*{1.5cm}\times\sum_{\ell=1}^N\left(
            \frac{(-1)^{j+r}}{(r-1)!}\binom{\ell-1}{j-1}\binom{N+\ell-r-1}{N-1}
        \right)
        \frac{(b-a)^{N+k-r}}{(N+k-r)!}
    \bigg]\\ \no
    &\quad=\frac{(N-1+j)!}{(b-a)^{j-k}}\frac{1}{(N-1+k)!}\\ \no &\hspace*{1.5cm}\times\sum_{\ell=1}^N\sum_{r=1}^{N}\left[
        \frac{(-1)^{j+r}(N-1+k)!}{(r-1)!(N+k-r)!}\binom{\ell-1}{j-1}\binom{N+\ell-r-1}{N-1}
    \right]\\ \no
    &\quad=\frac{(N-1+j)!}{(b-a)^{j-k}(N-1+k)!}\\ \no &\hspace*{1.5cm}\times\sum_{\ell=1}^N\binom{\ell-1}{j-1}\sum_{r=1}^{N}\left[
        (-1)^{j-\ell}(-1)^{\ell-r}\binom{N-1+k}{r-1}\binom{N+\ell-r-1}{\ell-r}
    \right].\no
\end{align}
Within the sum indexed by $r$, we observe that the first binomial coefficient is zero for $r\leq 0$, and the second binomial coefficient is zero whenever $\ell-r\leq0$. Since $\ell\leq N$, we see that for any $\ell$, the only possible nonzero terms within the sum are those such that $1\leq r\leq N$. Thus, the calculation may be continued as follows:
\begin{align}
    &\frac{(N-1+j)!}{(b-a)^{j-k}(N-1+k)!}\lb{4.79}\\ \no &\hspace*{1.5cm}\times\sum_{\ell=1}^N\binom{\ell-1}{j-1}\sum_{r=-\infty}^{\infty}\left[
        (-1)^{j-\ell}(-1)^{\ell-r}\binom{N-1+k}{r-1}\binom{N+\ell-r-1}{\ell-r}
    \right]\\\no
    &\quad=\frac{(N-1+j)!}{(b-a)^{j-k}(N-1+k)!}\\ \no &\hspace*{1.5cm}\times\sum_{\ell=1}^N\binom{\ell-1}{j-1}\sum_{r=-\infty}^{\infty}\left[
        (-1)^{j-\ell}\binom{N-1+k}{r-1}\binom{-N}{\ell-r}
    \right]\\\no
    &\quad=\frac{(N-1+j)!}{(b-a)^{j-k}(N-1+k)!}\lb{4.78}\\ \no &\hspace*{1.5cm}\times\sum_{\ell=1}^N(-1)^{j-\ell}\binom{\ell-1}{j-1}\sum_{r=-\infty}^{\infty}\left[
        \binom{N-1+k}{r-1}\binom{-N}{-N-\ell+r}
    \right]\\ \no
    &\quad=\frac{(N-1+j)!}{(b-a)^{j-k}(N-1+k)!}\\ \no &\hspace*{1.5cm}\times\sum_{\ell=1}^N\left[(-1)^{j-\ell}\binom{\ell-1}{j-1}\binom{N-1+k+(-N)}{N-1+k-(-1)+(-N-\ell)}
    \right]\\ \no
    &\quad=\frac{(N-1+j)!}{(b-a)^{j-k}(N-1+k)!}\sum_{\ell=1}^N\left[(-1)^{j-\ell}\binom{\ell-1}{j-1}\binom{k-1}{k-\ell}
    \right]\\ \no
    &\quad=\frac{(N-1+j)!}{(b-a)^{j-k}(N-1+k)!}\delta_{j,k}\\ \no
    &\quad=\delta_{j,k}\\
    &\quad=(I_N)_{j,k},\quad 1\leq j,k\leq N,\no
\end{align}
where the second, third and fifth equalities in \eqref{4.79} follow from Lemma \ref{L4.1} parts $(i)$, $(ii)$ and $(iii)$, respectively. Finally, for \eqref{4.62}:
\begin{align}
    & (\Phi_N(b)T_2)_{j,k}\\
    &\no\quad =
    \sum_{r=1}^N \frac{(b-a)^{N-r-k}}{(b-a)^{N-r+j}(N+k-r)!}
    \\
    &\no\hspace*{1.5cm}
    \times\sum_{\ell,s=1}^N
            \left(
            \frac{(-1)^{s+r}(N-1+s)!}{(r-1)!(s-j)!}\binom{\ell-1}{j-1}\binom{N+\ell-r-1}{N-1}
            \right)
    \\
    &\no \quad=
    (b-a)^{k-j} \frac{(N+k-1)!}{(N+k-1)!}
    \sum_{\ell, s=1}^N
            \Bigg(
            \frac{(-1)^{s+r}(N-1+s)!}{(s-j)!}\binom{\ell-1}{j-1}\\
    &\no \hspace*{6.19cm}
    \times\sum_{r=1}^N \binom{N+\ell-r-1}{N-1} \frac{1}{(r-1)!(N+k-r)!}
    \Bigg)
    \\ \no
    &\quad =
    (b-a)^{k-j}
    \sum_{\ell, s=1}^N
            \Bigg(
            \frac{(-1)^{s+\ell}(N-1+s)!}{(s-j)!(N+k-1)!}\binom{\ell-1}{j-1}\\
    &\no \hspace*{4.3cm}
    \times\sum_{r=1}^N (-1)^{\ell-r}\binom{N+k-1}{r-1}\binom{N+\ell-r-1}{\ell-r}
    \Bigg).\no
\end{align}
 Within the sum indexed by $r$, similar to the previous calculation, we observe that the first binomial coefficient is zero for $r\leq 0$, and the second binomial coefficient is zero whenever $\ell-r\leq0$. Since $\ell\leq N$, we see that for any $\ell$, the only possible nonzero terms within the sum are those such that $1\leq r\leq N$. Thus, the calculation may be continued as follows:
\begin{align}
    &
    (b-a)^{k-j}
    \sum_{\ell, s=1}^N
            \Bigg(
            \frac{(-1)^{s+\ell}(N-1+s)!}{(s-j)!(N+k-1)!}\binom{\ell-1}{j-1} \lb{4.81}\\
    &\no \hspace*{4.3cm}
    \times\sum_{r=-\infty}^{\infty} (-1)^{\ell-r}\binom{N+k-1}{r-1}\binom{N+\ell-r-1}{\ell-r}
    \Bigg)\\ \no
    &\quad =
    (b-a)^{k-j}
    \sum_{\ell, s=1}^N
            \left(
            \frac{(-1)^{s+\ell}(N-1+s)!}{(s-j)!(N+k-1)!}\binom{\ell-1}{j-1}
    \sum_{r=-\infty}^{\infty}
    \binom{N+k-1}{r-1}\binom{-N}{\ell-r}
    \right)
    \\
    &\no \quad =
    (b-a)^{k-j}
    \sum_{\ell, s=1}^N
            \Bigg(
            \frac{(-1)^{s+\ell}(N-1+s)!}{(s-j)!(N+k-1)!}\binom{\ell-1}{j-1}\\ 
    &\no \hspace*{5cm}
    \times\sum_{r=-\infty}^{\infty}
    \binom{N+k-1}{-1+r}\binom{-N}{(-N-\ell)+r}
    \Bigg)
    \\ \no
    &\quad =
    (b-a)^{k-j}
    \sum_{\ell, s=1}^N
            \left(
            \frac{(-1)^{s+\ell}(N-1+s)!}{(s-j)!(N+k-1)!}\binom{\ell-1}{j-1}
    \binom{k-1}{\ell-1}
    \right)
    \\ \no
    &\quad =
    (b-a)^{k-j}
    \sum_{s=1}^N
            \left(
            \frac{(N-1+s)!}{(s-j)!(N+k-1)!}
    \sum_{\ell=1}^N(-1)^{s+\ell}\binom{\ell-1}{j-1}\binom{k-1}{\ell-1}
    \right)
    \\ \no
    &\quad =
    (b-a)^{k-j}
    \sum_{s=1}^N
            \left(
            \frac{(N-1+s)!}{(s-j)!(N+k-1)!}
    \delta_{s,k}
    \right)
    \\ \no
    &\quad =
    \frac{(b-a)^{k-j}}{(k-j)!}
    \\
    &\no \quad =(T_1)_{j,k},\quad 1\leq j,k\leq N,
\end{align}
where the first, third, and fifth equalities in \eqref{4.81} use Lemma \ref{L4.1} parts $(i)$, $(ii)$, and $(iii)$, respectively.
 
By Proposition \ref{p3.5} we know that $B_{\rm K}$ is invertible, so by \eqref{4.60b}, we have shown that
\begin{equation}
    T_{\rm K}=B_{\rm K}^{-1}A_{\rm K}.
\end{equation}
Hence, we recover the following result due to Granovskyi and Oridoroga \cite{GO18a}:
\begin{theorem}[{\cite[Theorem 3.1]{GO18a}}]
    If $Z$ is defined by \eqref{4.24b}, then the domain of the Krein--von Neumann extension of $H_{Z,\min}$ is characterized by
        \begin{equation}\dom(H_{Z,\textrm{K}}) = \left\{y \in \dom(H_{Z,\max})\left|
        \begin{pmatrix}
        y(b) \\ y^{(1)}(b) \\ \vdots \\ y^{(2N-1)}(b)
        \end{pmatrix}=T_{\rm K}\begin{pmatrix}
        y(a) \\ y^{(1)}(a) \\ \vdots \\ y^{(2N-1)}(a)
        \end{pmatrix}\right.
        \right\},
        \end{equation}
    where $T_{\rm K}$ is the Toeplitz upper triangular matrix defined by \eqref{4.57}.
\end{theorem}

\begin{remark}
The above example reproduces, in full generality, the characterization of the Krein--von Neumann extension given in \cite{GO18a}. While \cite{GO18a} presents the  result with an elegant proof based on Taylor polynomials, our construction based on Theorem \ref{T3.4} and Proposition \ref{p3.5} requires the basis $\{\varphi_{k,1}\}_{k=1}^{2N}$ of the kernel of the maximal operator corresponding to Lemma \ref{L3.2}.  Since this basis may be of some independent interest, we have presented the complete details of its construction here.\hfill$\diamond$
\end{remark}

\medskip
 
\noindent {\bf Acknowledgments.}
The research of the authors was supported by the National Science Foundation under Grant DMS-1852288.


\end{document}